\documentclass[final]{elsarticle}
\usepackage[top=2.5cm, bottom=2.5cm, left=3.6cm, right=3.6cm]{geometry}

\usepackage{hyperref}
\usepackage{graphicx}
\usepackage{amsfonts}
\usepackage{amssymb}
\usepackage{amsmath}
\usepackage{amsthm}
\usepackage{float}
\usepackage{exscale}
\usepackage{relsize}
\usepackage{bm}
\usepackage{subeqnarray}
\usepackage{fancyhdr}
\usepackage{lastpage}
\usepackage{layout}
\usepackage{cleveref}
 
\newtheorem{thm}{Theorem}[section]
\newtheorem{corollary}[thm]{Corollary}

\newtheorem{lemma}[thm]{Lemma}

\newcommand{\bx}{\mathbf{x}}
\newcommand{\by}{\mathbf{y}}
\newcommand{\trace}{\mathrm{Tr}}
\newcommand{\MM}{\mathcal{M}}

\begin{document}

\begin{frontmatter}

\title{A GENERALIZED LIEB'S THEOREM \\ AND ITS APPLICATIONS TO SPECTRUM ESTIMATES \\ FOR A SUM OF RANDOM MATRICES}
\author{De Huang\fnref{myfootnote} }
\address{Applied and Computational Mathematics, California Institute of Technology, Pasadena, CA 91125, USA}
\fntext[myfootnote]{E-mail address:\ dhuang@caltech.edu.}
 
\begin{abstract}

In this paper we prove the concavity of the $k$-trace functions, $A\mapsto (\trace_k[\exp(H+\ln A)])^{1/k}$, on the convex cone of all positive definite matrices. $\trace_k[A]$ denotes the $k_{\mathrm{th}}$ elementary symmetric polynomial of the eigenvalues of $A$. As an application, we use the concavity of these $k$-trace functions to derive tail bounds and expectation estimates on the sum of the $k$ largest (or smallest) eigenvalues of a sum of random matrices.

\end{abstract}

\begin{keyword}
trace inequality, mixed discriminants, concavity of matrix functions, exterior algebra, random matrices, spectrum estimates.
\MSC[2010] 15A75, 15A15, 15A16, 15A42.
\end{keyword}

\end{frontmatter}

\section{Introduction}

Trace functions and trace inequalities have drawn great interests and are extremely useful in many fields, especially in quantum information theories \cite{guhr1998random,doi:10.1063/1.1666274,petz2007quantum}. In many related research, the concavity (or convexity) of some trace functions is one of the most studied topics. One celebrated achievement in this area is the Lieb's concavity theorem proved by Lieb \cite{LIEB1973267}, which states that the function
\begin{equation}
A\longmapsto \trace\big[K^*A^pKA^q\big], \quad p,q\geq 0, p+q\leq 1,
\end{equation}
is concave on the convex cone of all $n\times n$ Hermitian, positive semi-definite matrices, for arbitrary $K$ of the same size. Here $K^*$ is the conjugate adjoint of $K$. As an application, Lieb and Ruskai \cite{doi:10.1063/1.1666274} used the Lieb's concavity theorem to prove the strong subadditivity of quantum entropy. 

Among rich consequences of the Lieb's concavity theorem, a deep equivalent result, also established by Lieb \cite{LIEB1973267} and known as the Lieb's Theorem, is the concavity of the function
\begin{equation}
A\longmapsto \trace\big[\exp(H+\ln A)\big]
\end{equation}
on the convex cone of all $n\times n$ Hermitian, positive definite matrices, for arbitrary Hermitian matrix $H$ of the same size. Later, alternative proofs of the Lieb's Theorem (Carlen \cite{carlen2010trace}, Tropp \cite{MAL-048}) revealed its deep connections with quantum entropy and matrix tensors. With the help of the Lieb's theorem, Tropp \cite{Tropp2012,MAL-048} derived multiple important, user-friendly estimates, e.g. matrix master bounds and eigenvalue Chernoff bounds, that characterize the expectation and tail behaviors of extreme eigenvalues of random matrices of the form $Y=\sum_{i=1}^mX^{(i)}$, where $\{X^{(i)}\}_{1\leq i\leq m}$ is a finite sequence of independent, random, Hermitian matrices of the same size. Tropp et al. \cite{gittens2011tail} improved these results to interior eigenvalues by making use of the Courant--Fischer characterization of eigenvalues.

These estimates provide rich theoretical supports for studies and developments in stochastic models and algorithms for random matrices \cite{mehta2004random,wigner1993characteristic} in fields ranging from quantum physics \cite{guhr1998random} to financial statistics \cite{laloux2000random,plerou2002random}. A typical example is the study of clustering of random graphs \cite{Aiello:2000:RGM:335305.335326,1959} arising from research on social networks \cite{10.1007/978-3-540-77004-6_5}, image classification \cite{belkin2004regularization,bruna2013spectral} and so on. By spectral theory, the number of zero eigenvalues of the Laplacian of a graph indicates the number of connected components in the graph. A relaxed version is that the number of eigenvalues close to zero of the Laplacian of a graph indicates the number of major clusters in the graph. Based on this, many researchers have developed clustering methods by investigating the spectrum of graph Laplacians. When the graph is extremely large, the use of random sparsification or sampling is then critically necessary \cite{spielman2011graph,spielman2011spectral}. The practicability of these random approaches is guaranteed by expectation estimates and tail bounds of eigenvalues of random matrices as those in \cite{Tropp2012,MAL-048,gittens2011tail}. 

In many cases of interest the number of clusters is assumed \cite{chaudhuri2012spectral,NIPS2013_5099,rohe2011spectral} and so one may want to simultaneously study the behaviors of the $k$ smallest eigenvalues of the Laplacian of a random graph. Then a natural question is, can we generalize Tropp's estimates from the largest (or smallest) eigenvalue to the sum of the $k$ largest (or smallest) eigenvalues? Revisiting Tropp's proof of the master bounds in \cite{MAL-048}, we can see that this desired generalization actually requires a generalized version of the Lieb's theorem that the function 
\begin{equation}
\label{eqt:functions}
A\longmapsto \big(\trace_k\big[\exp(H+\ln A)\big]\big)^\frac{1}{k},\ \text{or equivalently}\ A\longmapsto \ln\trace_k\big[\exp(H+\ln A)\big],
\end{equation}
is concave on the convex cone of all $n\times n$ Hermitian, positive definite matrices, for arbitrary Hermitian matrix $H$ of the same size. Here $\trace_k(A)$ denotes the $k_{\mathrm{th}}$ elementary symmetric polynomial of the eigenvalues of $A$. Our main task of this paper is to prove this generalized Lieb's theorem.

The symmetric forms of eigenvalues in the functions \eqref{eqt:functions} bring our attention to theories of multilinear, symmetric forms of matrices. In particular, we will develop the proof by expressing the functions \eqref{eqt:functions} in terms of mixed discriminants or trace functions in exterior algebras. Furthermore, an essential step in our proof is due to the Alexandrov-Fenchel inequality for mixed discriminants, i.e.
\begin{equation}
D(A,B,A^{(3)},\cdots,A^{(n)})^2\geq D(A,A,A^{(3)},\cdots,A^{(n)})D(B,B,A^{(3)},\cdots,A^{(n)}),
\end{equation}
for any Hermitian matrix $B$ and any Hermitian, positive definite matrices $A,A^{(3)},\cdots,A^{(n)}$. The original Alexandrov-Fenchel inequality for mixed volumes of convex bodies, due to Alexandrov \cite{aleksandrov1937theory} and Fenchel \cite{fenchel1936inegalites} independently, is one of the deepest results in convex geometry. Alexandrov \cite{aleksandrov1938theory} then introduced
the notion of mixed discriminants of matrices and proved a variant of the Alexandrov-Fenchel inequality for mixed discriminants. This inequality was overlooked for a long time until it was applied to prove the Van der Waerden’s conjecture by Egorychev \cite{egorychev1981solution}. To see how our proof of the generalized Lieb's theorem may rely on the Alexandrov-Fenchel inequality, one can consider an extreme case by taking $H=0$. Then the concavity of $A\mapsto\big(\trace_k\big[A\big]\big)^\frac{1}{k}$ is due to the general Brunn-Minkowski theorem \cite{schneider2014convex}, which is a direct consequence of the Alexandrov-Fenchel inequality for mixed discriminants.

\subsection*{Outline} The rest of this paper is organized as follows. In \Cref{sec:Notations and main results}, we will introduce some notations and our main results. As preparation, we will review and discuss some basics and relevant results on mixed discriminants, exterior algebra and derivatives of matrix functions in \Cref{sec:Preparation}. \Cref{sec:Proofofmainthm} is devoted to the proofs of some important lemmas and the generalized Lieb's theorem (\Cref{thm:GeneralizedLiebThm}). In \Cref{sec:Application}, we will apply the generalized Lieb's theorem to derive expectation estimates and tail bounds for the sum of the k largest (or smallest) eigenvalues of a sum of random matrices, and compare these results to previous related works. 
 
\section{Notations and main results}
\label{sec:Notations and main results}
\subsection{General conventions}
For any positive integer $n$, we write $\mathbb{C}^n$ for the $n$-dimensional complex vector spaces equipped with the standard $l_2$ inner products, and $\mathbb{C}^{n\times n}$ for the space of all complex matrices of size $n\times n$. Let $\mathcal{H}_n$ be the linear space of all $n\times n$ Hermitian matrices, $\mathcal{H}_n^+$ be the convex cone of all $n\times n$ Hermitian, positive semi-definite matrices, and $\mathcal{H}_n^{++}$ be the convex cone of all $n\times n$ Hermitian, positive definite matrices. For any matrix $A\in \mathcal{H}_n$, we denote by $\lambda_i(A)$ the $i_{\text{th}}$ largest eigenvalue of $A$. We write $\bf{0}$ for square zero matrices (or operators) of suitable size according to the context, and $I_n$ for the identity matrix of size $n \times n$.

For any matrix $A\in\mathbb{C}^{n\times n}$ with eigenvalues $\lambda_1,\lambda_2,\cdots,\lambda_n$, we define the $k$-trace of $A$ to be 
\begin{equation}
\label{def:k-trace}
\trace_k[A] = \sum_{1\leq i_1<i_2<\cdots<i_k\leq n}\lambda_{i_1}\lambda_{i_2}\cdots\lambda_{i_k},\qquad 1\leq k\leq n.
\end{equation}
In particular, $\trace_1[A]=\trace[A]$ is the normal trace of $A$, and $\trace_n[A]=\det[A]$ is the determinant of $A$. If we write $A_{(i_1\cdots i_k,i_1\cdots i_k)}$ for the $k\times k$ principal submatrix of $A$ corresponding to the indices $i_1,i_2,\cdots,i_k$, then an equivalent definition of the $k$-trace of $A$ is given by 
\begin{equation}
\label{def:k-trace2}
\trace_k[A] = \sum_{1\leq i_1<i_2<\cdots<i_k\leq n}\det[A_{(i_1\cdots i_k,i_1\cdots i_k)}],\qquad 1\leq k\leq n.
\end{equation}
Using the second definition \eqref{def:k-trace2}, one can check that for any $1\leq k\leq n$, the $k$-trace enjoys the cyclic invariance property like the normal trace and the determinant. That is for any $A,B\in \mathbb{C}^n$, $\trace_k[AB]=\trace_k[BA]$.

For any function $f:\mathbb{R}\rightarrow\mathbb{R}$, the extension of $f$ to a function from $\mathcal{H}_n$ to $\mathcal{H}_n$ is given by 
\[f(A)=\sum_{i=1}^nf(\lambda_i)\mathbf{u}_i\mathbf{u}_i^*, \quad A\in \mathcal{H}_n,\]
where $\lambda_1,\lambda_2,\cdots,\lambda_n$ are the eigenvalues of $A$, and $\mathbf{u}_1,\mathbf{u}_2,\cdots,\mathbf{u}_n$ are the corresponding normalized eigenvectors. One can find more details and properties of matrix functions in \cite{carlen2010trace,Vershynina:2013}. 

\subsection{Main results}
Our main contribution is the following generalized Lieb's theorem.
\begin{thm}
\label{thm:GeneralizedLiebThm}
(\textbf{Generalized Lieb's Theorem}) For any $1\leq k\leq n$ and any $H\in\mathcal{H}_n$, the function
\begin{align}
\mathcal{H}_n^{++}\ &\longrightarrow \ \mathbb{R} \nonumber \\
A \ &\longmapsto\ \left(\trace_k [\exp(H+\ln A)]\right)^{\frac{1}{k}}
\label{eqt:function1}
\end{align}
is concave. Equivalently, for any $1\leq k\leq n$, the function
\begin{align}
\mathcal{H}_n^{++}\ &\longrightarrow \ \mathbb{R} \nonumber \\
A \ &\longmapsto\ \ln \trace_k [\exp(H+\ln A)]
\label{eqt:function2}
\end{align}
is concave.
\end{thm}

This theorem extends the Lieb's theorem from normal trace to elementary symmetric polynomials of eigenvalues, and hence connects it to theories of multilinear, symmetric forms of matrices. Indeed, as we will see in its proof, \Cref{thm:GeneralizedLiebThm} is a joint result of the original Lieb's theorem and the Alexandrov-Fenchel inequality for mixed discriminants. One can get some first ideas by looking at three extreme cases that relate to some well-known results.
\begin{itemize}
\item $k=1$: The concavity of $A\longmapsto\trace [\exp(H+\ln A)]$ is the original Lieb's theorem.
\item $k=n$: We have $\left(\trace_n [\exp(H+\ln A)]\right)^{\frac{1}{n}} = \det[A]^\frac{1}{n}\cdot\exp(\frac{1}{n}\trace[H])$ and $\ln \trace_n [\exp(H+\ln A)] = \ln\det[A]+\trace[H]$. The concavity of $\det[A]^\frac{1}{n}$ or $\ln\det[A]$ is known as the Brunn-Minkowski theorem \cite{schneider2014convex}.
\item $H=\bf{0}$: The concavity of $\trace_k[A]^\frac{1}{k}$, also know as the general Brunn-Minkowski theorem, is a consequence of the Alexandrov-Fenchel inequality for mixed discriminants. We will review this in \Cref{subsec:Mixeddiscriminant}. 
\end{itemize}
 
A direct application of our generalized Lieb's theorem is to derive expectation estimates and tail bounds on the sum of the k largest (or smallest) eigenvalues of a class of random matrices. In particular, we consider random matrices taking the form $Y=\sum_{i=1}^mX^{(i)}$, where $\{X^{(i)}\}_{1\leq i\leq m}\subset\mathcal{H}_n$ is a finite sequence of independent, random, Hermitian matrices. For this kind of matrices, we will prove the following generic estimates. Recall that we denote by $\lambda_i(A)$ the $i_{\text{th}}$ largest eigenvalue of any matrix $A\in\mathcal{H}_n$.

\begin{thm} 
\label{thm:ksummasterbound}
Given any finite sequence of independent, random matrices $\{X^{(i)}\}_{i=1}^m\subset\mathcal{H}_n$, let $Y=\sum_{i=1}^mX^{(i)}$. Then for any $1\leq k\leq n$,
\begin{subequations}
\begin{align}
&\sum_{i=1}^k\lambda_i(\mathbb{E}Y) \leq \mathbb{E}\sum_{i=1}^k\lambda_i(Y) \leq \inf_{\theta>0}\ \frac{1}{\theta}\ln \trace_k\big[\exp\big(\sum_{i=1}^m\ln\mathbb{E}\exp(\theta X^{(i)})\big)\big],
\label{eqt:masterexpect1}\\
&\sum_{i=1}^k\lambda_{n-i+1}(\mathbb{E}Y) \geq \mathbb{E}\sum_{i=1}^k\lambda_{n-i+1}(Y) \geq \sup_{\theta<0}\ \frac{1}{\theta}\ln \trace_k\big[\exp\big(\sum_{i=1}^m\ln\mathbb{E}\exp(\theta X^{(i)})\big)\big].
\label{eqt:masterexpect2}
\end{align}
\label{eqt:masterexpect}
\end{subequations}
Furthermore, for all $t\in \mathbb{R}$,
\begin{subequations}
\begin{align}
&\mathbb{P}\left\{\sum_{i=1}^k\lambda_i(Y) \geq  t\right\} \leq \inf_{\theta>0}\ e^{-\frac{\theta t}{k}}\Big(\trace_k\big[\exp\big(\sum_{i=1}^m\ln\mathbb{E}\exp(\theta X^{(i)})\big)\big]\Big)^\frac{1}{k},
\label{eqt:mastertail1}\\
& \mathbb{P}\left\{\sum_{i=1}^k\lambda_{n-i+1}(Y) \leq t\right\}\leq \inf_{\theta<0}\ e^{-\frac{\theta t}{k}}\Big(\trace_k\big[\exp\big(\sum_{i=1}^m\ln\mathbb{E}\exp(\theta X^{(i)})\big)\big]\Big)^\frac{1}{k}.
\label{eqt:mastertail2}
\end{align}
\label{eqt:mastertail}
\end{subequations}
\end{thm}

This theorem is a generalization of Theorem 3.6.1 in \cite{MAL-048}, where Tropp used matrix Laplace transform method and the Lieb's theorem to obtain the master bounds on the largest and the smallest eigenvalues for the same class of random matrices. The essential use of the Lieb's theorem in Tropp's proof is to establish the Jensen's inequality
\[\mathbb{E}\trace\big[\exp(H+\ln A)\big]\leq \trace\big[\exp(H+\ln \mathbb{E}A)\big],\]
for any random matrix $A\in \mathcal{H}_n^{++}$ and any fixed $H\in \mathcal{H}_n$. Using the generalized Lieb's theorem, we will extend this inequality to 
\[\mathbb{E}\big(\trace_k\big[\exp(H+\ln A)\big]\big)^\frac{1}{k}\leq \big(\trace_k\big[\exp(H+\ln \mathbb{E}A)\big]\big)^\frac{1}{k}\]
for proving tail bounds, and 
\[\mathbb{E}\ln\trace_k\big[\exp(H+\ln A)\big]\leq \ln\trace_k\big[\exp(H+\ln \mathbb{E}A)\big]\]
for proving expectation estimates. 

With \Cref{thm:ksummasterbound}, we can establish more concrete estimates for particular random matrices in this class. For example, we consider the scenario where each $X^{(i)}$ in the sum $Y=\sum_{i=1}^mX^{(i)}$ also satisfies $0\leq \lambda_{n}(X^{(i)})\leq \lambda_{1}(X^{(i)})\leq c$ for some uniform constant $c>0$. For this positive semi-definite case, we will prove the so called eigenvalue Chernoff bounds, which again generalize Theorem 5.1.1 \cite{MAL-048} from the largest (or smallest) eigenvalue to the sum of the $k$ largest (or smallest) eigenvalues.

\section{Preparations}
\label{sec:Preparation}

\subsection{Mixed discriminant}
\label{subsec:Mixeddiscriminant} 
The mixed discriminant $D(A^{(1)},A^{(2)},\cdots,A^{(n)})$ of $n$ matrices $A^{(1)},A^{(2)},\cdots,A^{(n)}\in \mathbb{C}^{n\times n}$ is defined as 
\begin{equation}
D(A^{(1)},A^{(2)},\cdots,A^{(n)}) = \frac{1}{n!} \sum_{\sigma\in S_n}\det\left[\begin{array}{cccc}
 A_{11}^{(\sigma(1))} & A_{12}^{(\sigma(2))} & \cdots & A_{1n}^{(\sigma(n))} \\
 A_{21}^{(\sigma(1))} & A_{22}^{(\sigma(2))} & \cdots & A_{1n}^{(\sigma(n))} \\
 \vdots & \vdots & \ddots & \vdots\\
 A_{n1}^{(\sigma(1))} & A_{n2}^{(\sigma(2))} & \cdots & A_{nn}^{(\sigma(n))}
\end{array}\right],
\label{eqt:mixdiscriminant}
\end{equation}
where $S_n$ denotes the symmetric group of order $n$. We here list some basic facts about mixed discriminants. For more properties of mixed discriminants, one may refer to \cite{bapat1989mixed,panov1987some}.
\begin{itemize} 
\item Symmetry: $D(A^{(1)},A^{(2)},\cdots,A^{(n)})$ is symmetric in $A^{(1)},A^{(2)},\cdots,A^{(n)}$, i.e. 
\[D(A^{(1)},A^{(2)},\cdots,A^{(n)}) = D(A^{\sigma(1)},A^{\sigma(2)},\cdots,A^{\sigma(n)}),\quad \sigma\in S_n.\]
\item Multilinearity: for any $\alpha,\beta\in \mathbb{R}$, 
\[D(\alpha A+\beta B,A^{(2)},\cdots,A^{(n)}) = \alpha D(A,A^{(2)},\cdots,A^{(n)}) + \beta D(B,A^{(2)},\cdots,A^{(n)}).\]
\item Positiveness \cite{bapat1989mixed}: If $A^{(1)},A^{(2)},\cdots,A^{(n)}\in \mathcal{H}_n^+$, then $D(A^{(1)},A^{(2)},\cdots,A^{(n)})\geq0$; if $A^{(1)},A^{(2)},\cdots,A^{(n)}\in \mathcal{H}_n^{++}$, then $D(A^{(1)},A^{(2)},\cdots,A^{(n)})>0$.
\end{itemize}

The relation between the mixed discriminant and $\trace_k$ is obvious. If we calculate the mixed discriminant for $k$ copies of $A\in\mathbb{C}^{n\times n}$ and $n-k$ copies of $I_n$, we can find that 
\begin{align}
D(\underbrace{A,\cdots,A}_{k},\underbrace{I_n,\cdots,I_n}_{n-k}) = \binom{n}{k}^{-1}\sum_{1\leq i_1<i_2<\cdots<i_k\leq n}\det[A_{(i_1\cdots i_k,i_1\cdots i_k)}] = \binom{n}{k}^{-1}\trace_k[A].
\end{align}
This is why the mixed discriminant plays an important role in the proof of our main theorem. In particular, we will need the following inequality on mixed discriminant by Alexandrov \cite{aleksandrov1938theory}.

\begin{thm}
\label{thm:AFinequality}
(\textbf{Alexandrov-Fenchel Inequality for Mixed Discriminants})
For any $B\in \mathcal{H}_n$ and any $A,\underbrace{A^{(3)},\cdots,A^{(n)}}_{n-2}\in \mathcal{H}_n^{++}$, we have
\begin{equation}
D(A,B,A^{(3)},\cdots,A^{(n)})^2\geq D(A,A,A^{(3)},\cdots,A^{(n)})D(B,B,A^{(3)},\cdots,A^{(n)}),
\label{eqt:AFinequality}
\end{equation}
with equality if and only if $B=\lambda A$ for some $\lambda\in \mathbb{R}$.
\end{thm}
This theorem originally applied to real symmetric matrices when established. A proof of its extension to Hermitian matrices can be found in \cite{2017arXiv171000520L}. By continuity, inequality \eqref{eqt:AFinequality} can extend to the case that $A,A^{(3)},\cdots,A^{(n)}\in \mathcal{H}_n^{+}$, but the necessity of the condition for equality is no longer valid.

Repeatedly applying the Alexandrov-Fenchel inequality \eqref{eqt:AFinequality} grants us the following corollary.
\begin{corollary}
\label{cor:generalAFinequality}
For any $0\leq l\leq k\leq n$, and any $A,B,\underbrace{A^{(k+1)},\cdots,A^{(n)}}_{n-k}\in \mathcal{H}_n^{+}$, we have
\begin{align}
&\ D(\underbrace{A,\cdots,A}_{l},\underbrace{B,\cdots,B}_{k-l},\underbrace{A^{(k+1)},\cdots,A^{(n)}}_{n-k})^k\\
\geq&\ D(\underbrace{A,\cdots,A}_{k},\underbrace{A^{(k+1)},\cdots,A^{(n)}}_{n-k})^l\cdot D(\underbrace{B,\cdots,B}_{k},\underbrace{A^{(k+1)},\cdots,A^{(n)}}_{n-k})^{k-l}. \nonumber
\end{align}
\end{corollary}

A direct result of \Cref{cor:generalAFinequality} is the following general Brunn-Minkowski theorem for mixed discriminants.

\begin{corollary}
\label{cor:BMtheorem}
(\textbf{General Brunn-Minkowski Theorem for Mixed Discriminants})
For any $1\leq k\leq n$, and any fixed $\underbrace{A^{(k+1)},\cdots,A^{(n)}}_{n-k}\in \mathcal{H}_n^{+}$, the function
\begin{align}
\mathcal{H}_n^{+}\ &\longrightarrow \ \mathbb{R} \nonumber \\
A \ &\longmapsto\ D(\underbrace{A,\cdots,A}_{k},\underbrace{A^{(k+1)},\cdots,A^{(n)}}_{n-k})^\frac{1}{k}
\end{align}
is concave.
\end{corollary}
\begin{proof}
Fixing $A^{(k+1)},\cdots,A^{(n)}$, we will use $D(A[k])$ and $D(A[l],B[k-l])$ to denote 
\[D(\underbrace{A,\cdots,A}_{k},\underbrace{A^{(k+1)},\cdots,A^{(n)}}_{n-k})^\frac{1}{k}\ \text{and}\ D(\underbrace{A,\cdots,A}_{l},\underbrace{B,\cdots,B}_{k-l},\underbrace{A^{(k+1)},\cdots,A^{(n)}}_{n-k})\]
respectively. For any $A,B\in \mathcal{H}_n^+$, and any $\tau\in [0,1]$, using the multilinearity of mixed discriminants and \Cref{cor:generalAFinequality}, we have
\begin{align*}
D((\tau A+(1-\tau)B)[k]) =&\ \sum_{l=0}^k\binom{k}{l}\tau^l(1-\tau)^{k-l}D(A[l],B[k-l])\\
\geq&\ \sum_{l=0}^k\binom{k}{l}\tau^l(1-\tau)^{k-l}D(A[k])^\frac{l}{k}D(B[k])^\frac{k-l}{k}\\
=&\ \big(\tau D(A[k])^\frac{1}{k}+(1-\tau)D(B[k])^\frac{1}{k}\big)^k,
\end{align*}
that is $D((\tau A+(1-\tau)B)[k])^\frac{1}{k}\geq \tau D(A[k])^\frac{1}{k}+(1-\tau)D(B[k])^\frac{1}{k}$.
\end{proof}

If we choose $A^{(k+1)},\cdots,A^{(n)}$ to be $n-k$ copies of $I_n$, \Cref{cor:BMtheorem} immediately implies that the function $A\mapsto \big(\trace_k\big[A\big]\big)^\frac{1}{k}$ is concave on $\mathcal{H}_n^+$, which is a special case of \Cref{thm:GeneralizedLiebThm} with $H=\bf{0}$. So we see the connection between the Alexandrov-Fenchel inequality and our generalized Lieb's theorem. However, the arguments in the proof of \Cref{cor:BMtheorem} do not seem to work with $H\neq \bf{0}$. We hence need more tools to handle the more general case.

\subsection{Exterior algebra}
Here we give a brief review of exterior algebras on the vector space $\mathbb{C}^n$. For more details, one may refer to \cite{bishop1980tensor,Rotman:2114272}. For the convenience of our use, the notations in our paper might be different from those in other materials. For any $1\leq k\leq n$, let $\wedge^k(\mathbb{C}^n)$ denote the vector space of the $k_{th}$ exterior algebra of $\mathbb{C}^n$, equipped with the inner product 
\begin{align*}
\langle \cdot,\cdot\rangle_{\wedge^k}:\quad \wedge^k(\mathbb{C}^n)\times\wedge^k(\mathbb{C}^n)\ &\longrightarrow\ \mathbb{C}\\
\langle u_1\wedge\cdots\wedge u_k,v_1\wedge\cdots\wedge v_k\rangle_{\wedge^k}\ &= 
\det\left[\begin{array}{cccc}
 \langle u_1,v_1\rangle & \langle u_1,v_2\rangle  & \cdots & \langle u_1,v_k\rangle \\
 \langle u_2,v_1\rangle & \langle u_2,v_2\rangle & \cdots & \langle u_2,v_k\rangle\\
 \vdots & \vdots & \ddots & \vdots \\
 \langle u_k,v_1\rangle & \langle u_k,v_2\rangle & \cdots & \langle u_k,v_k\rangle
\end{array}\right],
\end{align*}
where $\langle u,v\rangle=u^*v$ is the standard $l_2$ inner product on $\mathbb{C}^n$.

Let $\mathcal{L}(\wedge^k(\mathbb{C}^n))$ denote the space of all linear operators from $\wedge^k(\mathbb{C}^n)$ to itself. For any matrices $A^{(1)},A^{(2)},\cdots,A^{(k)}\in \mathbb{C}^{n\times n}$, we can define an element in $\mathcal{L}(\wedge^k(\mathbb{C}^n))$:
\begin{align}
\label{def:ExteriorOps}
\mathcal{M}^{(k)}(A^{(1)},A^{(2)},\cdots,A^{(k)}):\quad \wedge^k(\mathbb{C}^n) &\ \longrightarrow\ \wedge^k(\mathbb{C}^n)\nonumber \\
v_1\wedge v_2\wedge\cdots\wedge v_k &\ \longmapsto\  \sum_{ \sigma\in S_k} A^{(\sigma(1))}v_1\wedge A^{(\sigma(2))}v_2\wedge\cdots\wedge A^{(\sigma(k))}v_k,
\end{align}
where $S_k$ is the symmetric group of order $k$. Apparently, the map $\mathcal{M}^{(k)}(A^{(1)},A^{(2)},\cdots,A^{(k)})$ is symmetric in $A^{(1)},A^{(2)},\cdots,A^{(k)}$, and linear in each single $A^{(i)}$. For simplicity, we will use the following notations for any matrices $A,B,C\in \mathbb{C}^{n\times n}$:
\begin{subequations}
\label{eqt:threeops}
\begin{align}
\MM_0^{(k)}(A) &= \frac{1}{k!}\mathcal{M}^{(k)}(A,\cdots,A),\\
\MM_1^{(k)}(A;B) &= \frac{1}{(k-1)!}\mathcal{M}^{(k)}(A,B,\cdots,B),\\
\MM_2^{(k)}(A,B;C) &= \frac{1}{(k-2)!}\mathcal{M}^{(k)}(A,B,C,\cdots,C).
\end{align}
\end{subequations}
To avoid confusion, we define $\MM_1^{(1)}(A;B)=\MM_0^{1}(A)$, $\MM_2^{(1)}(A,B;C)=\bf{0}$, and $\MM_2^{(2)}(A,B;C)=\MM_1^{(2)}(A;B)$. Obviously the identity operator in $\mathcal{L}(\wedge^k(\mathbb{C}^n))$ is $\MM_0(I_n)$. We will be using the following properties:
\begin{itemize}
\item Invertibility: if $A\in \mathbb{C}^{n\times n}$ is invertible, then $(\MM_0^{(k)}(A))^{-1} = \MM_0^{(k)}(A^{-1})$.
\item Adjoint: for any $A\in \mathbb{C}^{n\times n}$, $(\MM_0^{(k)}(A))^*=\MM_0^{(k)}(A^*)$, with respect to the inner product $\langle \cdot,\cdot\rangle_{\wedge^k}$. 
\item Positiveness: If $A\in \mathcal{H}_n$, then $\MM_0^{(k)}(A)$ is Hermitian; if $A\in \mathcal{H}_n^{+}$, then $\MM_0^{(k)}(A)\succeq \mathbf{0}$; if $A\in \mathcal{H}_n^{++}$, then $\MM_0^{(k)}(A)\succ\mathbf{0}$. 

\item Product properties: for any $A,B,C,D\in \mathbb{C}^{n\times n}$, we have
\begin{subequations}
\label{eqt:prodproperty}
\begin{align}
\MM_0^{(k)}(AB) & = \MM_0^{(k)}(A)\MM_0^{(k)}(B),\\
\MM_1^{(k)}(A;B)\MM_0^{(k)}(C) &= \MM_1^{(k)}(AC;BC),\\ 
\MM_0^{(k)}(C)\MM_1^{(k)}(A;B) &= \MM_1^{(k)}(CA;CB),\\
\MM_1^{(k)}(A;C)\MM_1^{(k)}(B;D) &= \MM_2^{(k)}(AD,CB;CD) + \MM_1^{(k)}(AB;CD).
\end{align}
\end{subequations}

\item Derivative properties: for any differentiable functions $A(t),B(t):\mathbb{R}\longrightarrow\mathbb{C}^{n\times n}$, we have
\begin{subequations}
\label{eqt:deriproperty}
\begin{align}
\frac{\partial}{\partial t}\MM_0^{(k)}(A(t)) & = \MM_1^{(k)}(A'(t);A(t))\\
\frac{\partial}{\partial t}\MM_1^{(k)}(A(t);B(t)) &= \MM_1^{(k)}(A'(t);B(t)) + \MM_2^{(k)}(A(t),B'(t);B(t)).
\end{align}
\end{subequations}

\end{itemize}

Next we consider the natural basis of $\wedge^k(\mathbb{C}^n)$, 
\[\{e_{i_1}\wedge e_{i_2}\wedge\cdots\wedge e_{i_k} \}_{1\leq i_1<i_2<\cdots<i_k\leq n},\]
which is orthogonal under the inner product $\langle \cdot,\cdot\rangle_{\wedge^k}$. 
Then the trace function on $\mathcal{L}(\wedge^k(\mathbb{C}^n))$ is defined as 
\begin{align}
\trace:\quad \mathcal{L}(\wedge^k(\mathbb{C}^n))\ &\longrightarrow\ \mathbb{C}\nonumber\\
\trace\big[\mathcal{F} \big]\ &= \ \sum_{1\leq i_1<i_2<\cdots<i_k\leq n} \langle e_{i_1}\wedge e_{i_2}\wedge\cdots\wedge e_{i_k},\mathcal{F}(e_{i_1}\wedge e_{i_2}\wedge\cdots\wedge e_{i_k}) \rangle_{\wedge^k}.
\end{align}
It is not hard to check that this trace function is also invariant under cyclic permutation, i.e. $\trace\big[\mathcal{F}\mathcal{G}\big]=\trace\big[\mathcal{G}\mathcal{F}\big]$ for any $\mathcal{F},\mathcal{G}\in \mathcal{L}(\wedge^k(\mathbb{C}^n))$. Then for any $A^{(1)},\cdots,A^{(k)}\in \mathbb{C}^{n\times n}$, the trace $\trace[\mathcal{M}^{(k)}(A^{(1)},\cdots,A^{(k)})]$ coincides with the definition of the mixed discriminant, as one can check that 
\begin{align}
\trace\big[\MM^{(k)}(A^{(1)},\cdots,A^{(k)})\big]=&\ \sum_{\sigma\in S_k}\sum_{1\leq i_1<\cdots<i_k\leq n} \langle e_{i_1}\wedge\cdots\wedge e_{i_k},A^{(\sigma(1))}e_{i_1}\wedge\cdots\wedge A^{(\sigma(k))}e_{i_k} \rangle_{\wedge^k}\nonumber\\
=&\ \frac{n!}{(n-k)!}D(A^{(1)},\cdots,A^{(k)},\underbrace{I_n,\cdots,I_n}_{n-k}).
\label{eqt:extrace2mixdiscriminant}
\end{align}
From this observation, we can now express the $k$-trace of a matrix $A\in \mathbb{C}^{n\times n}$ as 
\begin{equation}
\trace_k[A] = \trace\big[\MM_0^{(k)}(A)\big].
\label{eqt:extrace2ktrace}
\end{equation} 
For those who are familiar with exterior algebra, it is clear that the spectrum of $\MM_0^{(k)}$ is just $\{\lambda_{i_1}\lambda_{i_2}\cdots\lambda_{i_k}\}_{1\leq i_1<i_2<\cdots<i_k\leq n}$, where $\lambda_1,\lambda_2,\cdots,\lambda_n$ are the eigenvalues of $A$. So in this way it is more convenient to see that $\trace\big[\MM_0^{(k)}(A)\big] = \mathrm{sum}(\text{spectrum of }\MM_0^{(k)}(A)) = \sum_{1\leq i_1<\cdots<i_k\leq n}\lambda_{i_1}\lambda_{i_2}\cdots\lambda_{i_k} =\trace_k[A]$. Our proof of \Cref{thm:GeneralizedLiebThm} will base on the expression \eqref{eqt:extrace2ktrace}. 

In fact, our proof the main theorem can be done without introducing the exterior algebra. We can instead go through the whole proof only using notations of mixed discriminant. The advantage of using exterior algebra is that it interprets the $k$-trace as the normal trace of operators in a space of higher dimension, so our $k$-trace functions have a nicer form that imitates the trace function in the original Lieb's theorem. Also for the same reason, we are able to construct our proof by following the arguments of Lieb's original proof in \cite{LIEB1973267}.

We next introduce some notations to simplify the expressions in what follows. For any $n$ real numbers $\lambda_1,\lambda_2,\cdots,\lambda_n\in \mathbb{R}$, we define the three symmetric forms 
\begin{subequations}
\label{eqt:sympolys}
\begin{align}
p^{(n,k)} =&\ \sum_{1\leq i_1<i_2<\cdots<i_k\leq n}\lambda_{i_1}\lambda_{i_2}\cdots\lambda_{i_k},\quad 1\leq k\leq n,\\
d_i^{(n,k)} =&\ \sum_{\begin{subarray}{c}1\leq j_1<j_2<\cdots<j_{k-1}\leq n\\ i\notin \{j_1,j_2,\cdots,j_{k-1}\}\end{subarray}}\lambda_{j_1}\lambda_{j_2}\cdots\lambda_{j_{k-1}},\quad 2\leq k\leq n,\ 1\leq i\leq n, \\
g_{ij}^{(n,k)} =&\ \sum_{\begin{subarray}{c}1\leq l_1<l_2<\cdots<l_{k-2}\leq n\\ i,j\notin \{l_1,l_2,\cdots,l_{k-2}\}\end{subarray}}\lambda_{l_1}\lambda_{l_2}\cdots\lambda_{l_{k-2}},\quad 3\leq k\leq n,\ 1\leq i,j\leq n,\ i\neq j.
\end{align}
\end{subequations}
For consistency, we define $d^{(n,k)}_i=1$ if $k=1$; $g_{ij}^{(n,k)}=1$ if $k=2$ and $i\neq j$; $g_{ij}^{(n,k)}=0$ if $k=1$ or $i=j$. Also we define $p^{(n,k)}=d^{(n,k)}_i=g_{ij}^{(n,k)}=0$ if $k>n$. Throughout this paper, whenever we are given some real numbers $\lambda_1,\lambda_2,\cdots,\lambda_n$, the quantities $p^{(n,k)},d^{(n,k)}_i,g_{ij}^{(n,k)}$ are always defined correspondingly with respect to $\{\lambda_i\}_{1\leq i\leq n}$. The following relations are easy to verify with the definitions above, and will be useful in our proofs of lemmas and theorems. For any $n,k$, and any $1\leq i,j\leq n$ such that $i\neq j$, we have the expansion relations
\begin{equation}
p^{(n,k)} = \lambda_id_i^{(n,k)} + d_i^{(n,k+1)}, \qquad d_i^{(n,k)} = \lambda_jg_{ij}^{(n,k)}+g_{ij}^{(n,k+1)}.
\label{eqt:expansion}
\end{equation}

With the notations defined above, we give the following lemma. The proof is straightforward by definition, so we omit it here.

\begin{lemma} 
\label{lem:impidentities}
For any $A,B\in \mathbb{C}^{n\times n}$, and any diagonal matrix $\Lambda\in\mathbb{C}^{n\times n}$ with diagonal entries $\lambda_1,\lambda_2,\cdots,\lambda_n$, we have the following identities 
\begin{subequations}
\label{eqt:impidentities}
\begin{align}
\trace\big[\MM_0^{(k)}(\Lambda)\big] &=\ p^{(n,k)},\\
\trace\big[\MM_1^{(k)}(A;\Lambda)\big] &=\ \sum_{i=1}^nA_{ii}d_i^{(n,k)},\\
\trace\big[\MM_2^{(k)}(A,B;\Lambda)\big] &=\ \sum_{1\leq i,j\leq n} (A_{ii}B_{jj}-A_{ji}B_{ij})g_{ij}^{(n,k)},
\end{align}
\end{subequations}
for all $1\leq k\leq n$, where $p^{(n,k)},d_i^{(n,k)},g_{ij}^{(n,k)}$ are defined with respect to $\lambda_1,\lambda_2,\cdots,\lambda_n$.
\end{lemma}

\subsection{Derivatives of some matrix functions}

Let us remind ourselves that a basic but important way to prove concavity of a differentiable function $f(t)$ is by showing that $f''(t)\leq 0$. Similarly, one way to prove concavity of a differentiable multivariate function $f(\bx)$ is by showing that the second directional derivative $\frac{\partial^2}{\partial t^2}f(\bx+t\by)|_{t=0}\leq 0$ for all allowed direction $\by$. We will use this idea to prove the concavity of the $k$-trace functions \eqref{eqt:function1} and \eqref{eqt:function2}. For this purpose, we would need the following matrix derivative formulas.
\begin{itemize}
\item Consider a function $A(t):(a,b)\longrightarrow \mathcal{H}_n$, such that $A(t)$ is differentiable on $(a,b)$, then we have\textsuperscript{\cite{doi:10.1063/1.1705306}}
\begin{equation}
\label{eqt:expderivative}
\frac{\partial}{\partial t}\exp\big(A(t)\big) = \int_0^1\exp\big(sA(t)\big)A'(t)\exp\big((1-s)A(t)\big)ds.
\end{equation}
$A'(t)$ denotes the derivative of $A(t)$ with respect to $t$.
\item Consider a function $A(t):(a,b)\longrightarrow \mathcal{H}_n^{++}$, such that $A(t)$ is differentiable on $(a,b)$, then we have\textsuperscript{\cite{LIEB1973267}}
\begin{equation}
\label{eqt:inversederivative}
\frac{\partial}{\partial t} \big(A(t)\big)^{-1} = -\big(A(t)\big)^{-1}A'(t)\big(A(t)\big)^{-1},
\end{equation}
and 
\begin{equation}
\label{eqt:logderivative}
\frac{\partial}{\partial t} \ln\big(A(t)\big) = \int_0^{\infty}\big(A(t)+\tau I_n\big)^{-1}A'(t)\big(A(t)+\tau I_n\big)^{-1}d\tau.
\end{equation}
\end{itemize} 

\section{Proof of the generalized Lieb's theorem}
\label{sec:Proofofmainthm}
As mentioned before, our generalized Lieb's theorem is a joint result of the original Lieb's theorem and the Alexandrov-Fenchel inequality. But we will not use the Lieb's theorem directly. Instead, we will be using the following lemma, also due to Lieb \cite{LIEB1973267}, which is an equivalence of the Lieb's theorem. We provide the proof here only to show its connection to the Lieb's theorem.

\begin{lemma}
\label{lem:LiebThm}
Given any $A\in\mathcal{H}_n^{++}$, $C\in\mathcal{H}_n$, define 
\[T=\int_0^\infty(A+\tau I)^{-1}C(A+\tau I)^{-1}d\tau,\]
\[R=2\int_0^\infty(A+\tau I)^{-1}C(A+\tau I)^{-1}C(A+\tau I)^{-1}d\tau,\]
then for any $B\in\mathcal{H}_n^{+}$, we have
\begin{equation}
\label{eqt:LiebThm}
\int_0^1ds\trace\big[TB^sTB^{1-s}\big] - \trace\big[RB\big]\leq0.\end{equation}
\end{lemma}

\begin{proof}
By Lieb's theorem (Theorem 6 \cite{LIEB1973267}), for any $H\in\mathcal{H}_n$, the function $g(t)=\trace\big[\exp(H+\ln(A+tC))\big]$ is concave. Also this function is smooth in $t$ for $t$ small enough such that $A+tC\in\mathcal{H}_n^{++}$. Thus we have $\frac{\partial^2}{\partial t^2}g(t)|_{t=0}=g''(0)\leq 0$. Write $B(t)=\exp(H+\ln(A+tC))$, and 
\[T(t) = \int_0^\infty(A+tC+\tau I)^{-1}C(A+tC+\tau I)^{-1}d\tau,\]
\[R(t) = 2\int_0^\infty(A+tC+\tau I)^{-1}C(A+tC+\tau I)^{-1}C(A+tC+\tau I)^{-1}d\tau.\] 
It is easy to check that $\frac{\partial}{\partial t}\ln(A+tC)=T(t)$, $T'(t)=-R(t)$ by formulas \eqref{eqt:inversederivative} and \eqref{eqt:logderivative}. Then using the derivative formulas \eqref{eqt:expderivative}, \eqref{eqt:inversederivative} and \eqref{eqt:logderivative}, we have
\begin{align*}{}
g'(t) = \int_0^1ds\trace\big[(B(t))^sT(t)(B(t))^{1-s}\big]= \trace\big[T(t)B(t)\big],
\end{align*}
and
\[g''(t) = \trace\big[T'(t)B(t)\big] + \int_0^1ds\trace\big[T(t)(B(t))^sT(t)(B(t))^{1-s}\big].\]
For any $B\in \mathcal{H}_n^{++}$, we may choose $H=\ln B-\ln A$, so that $B(0)=\exp(H+\ln A)=B$. And notice that $T(0)=T,R(0)=R$, we thus have
\[-\trace\big[RB\big]+\int_0^1ds\trace\big[TB^sTB^{1-s}\big]=g''(0)\leq0.\]
The extension to $B\in \mathcal{H}_n^{+}$ can be done by continuity.
\end{proof}

We use this variant of the Lieb's theorem since it is more convenient for us to choose arbitrary $B\in\mathcal{H}_n^{++}$ in inequality \eqref{eqt:LiebThm}. In particular, if we choose $B$ to be diagonal with diagonal entries $b_1,b_2,\cdots,b_n$, then \Cref{lem:LiebThm} implies that 
\begin{equation}
\label{eqt:LiebThm2}
\sum_{i=1}^nR_{ii}b_i\geq \int_0^1ds\sum_{i=1}^n\sum_{j=1}^nT_{ij}b_j^sT_{ji}b_i^{1-s},
\end{equation}
which is a critical estimate that we will be using.

We now prove a trace inequalities using \Cref{lem:LiebThm} and the Alexandrov-Fenchel inequality \Cref{thm:AFinequality}. This inequality can be seen as a generalization of \Cref{lem:LiebThm} from $k=1$ to all $1\leq k\leq n$.

\begin{lemma} 
\label{lem:TraceIneq}
For arbitrary $A\in \mathcal{H}_n^{++},B\in \mathcal{H}_n^{+},C\in\mathcal{H}_n$, let
\[T = \int_0^\infty(A+\tau I)^{-1}C(A+\tau I)^{-1}d\tau,\]
\[R = 2\int_0^\infty(A+\tau I)^{-1}C(A+\tau I)^{-1}C(A+\tau I)^{-1}d\tau,\] 
then we have, for all $1\leq k\leq n$,
\begin{align}
&\ \int_0^1ds\trace\big[\MM_1^{(k)}(TB^s;B^s)\MM_1^{(k)}(TB^{1-s};B^{1-s})\big] -\trace\big[\MM_1^{(k)}(RB;B)\big] \leq \trace\big[\MM_2^{(k)}(TB,TB,B)\big]. \label{eqt:TraceIneq}
\end{align}

\end{lemma}

\begin{proof} 
We first claim that we only need to consider the case when $B=\Lambda$ is a diagonal matrix with all diagonal entries $\lambda_1,\lambda_2,\cdots,\lambda_n\geq0$. Indeed, if $B$ is not diagonal, we consider its eigenvalue decomposition $B=U\Lambda U^T$, where $U\in\mathbb{C}^{n \times n}$ is unitary, and $\Lambda$ is a diagonal matrix whose diagonal entries $\lambda_1,\lambda_2,\cdots,\lambda_n$ are the eigenvalues of $B$. Since $B\in \mathcal{H}_n^{+}$, $\lambda_1,\lambda_2,\cdots,\lambda_n\geq 0$. If we introduce $\widetilde{A}=U^TAU,\widetilde{C}=U^TCU,\widetilde{T}=U^TTU,\widetilde{R}=U^TRU$, we have
\[\widetilde{T} = \int_0^\infty(\widetilde{A}+\tau I)^{-1}\widetilde{C}(\widetilde{A}+\tau I)^{-1}d\tau,\]
\[\widetilde{R} = 2\int_0^\infty(\widetilde{A}+\tau I)^{-1}\widetilde{C}(\widetilde{A}+\tau I)^{-1}\widetilde{C}(\widetilde{A}+\tau I)^{-1}d\tau.\] 
Then using the  cyclic invariance of trace and the product properties \eqref{eqt:prodproperty}, we have, for example,
\begin{align*}
&\ \trace\big[\MM_1^{(k)}(TB^s;B^s)\MM_1^{(k)}(TB^{1-s};B^{1-s})\big]\\
=&\ \trace\big[\MM_1^{(k)}(UU^TTU\Lambda^sU^T;U\Lambda^sU^T)\MM_1^{(k)}(UU^TTU\Lambda^{1-s}U^T;U\Lambda^{1-s}U^T)\big]\\
=&\ \trace\big[\MM_0^{(k)}(U)\MM_1^{(k)}(\widetilde{T}\Lambda^s;\Lambda^s)\MM_0^{(k)}(U^T)\MM_0^{(k)}(U)\MM_1^{(k)}(\widetilde{T}\Lambda^{1-s};\Lambda^{1-s})\MM_0^{(k)}(U^T)\big]\\
=&\ \trace\big[\MM_1^{(k)}(\widetilde{T}\Lambda^s;\Lambda^s)\MM_0^{(k)}(U^T)\MM_0^{(k)}(U)\MM_1^{(k)}(\widetilde{T}\Lambda^{1-s};\Lambda^{1-s})\MM_0^{(k)}(U^T)\MM_0^{(k)}(U)\big]\\
=&\ \trace\big[\MM_1^{(k)}(\widetilde{T}\Lambda^s;\Lambda^s)\MM_1^{(k)}(\widetilde{T}\Lambda^{1-s};\Lambda^{1-s})\big].
\end{align*}
Using the same trick to the other terms in the inequalities \eqref{eqt:TraceIneq}, one can show that \eqref{eqt:TraceIneq} is equivalent to 
\begin{equation*}
\int_0^1ds\trace\big[\MM_1^{(k)}(\widetilde{T}\Lambda^s;\Lambda^s)\MM_1^{(k)}(\widetilde{T}\Lambda^{1-s};\Lambda^{1-s})\big]-\trace\big[\MM_1^{(k)}(\widetilde{R}\Lambda;\Lambda)\big] \leq \trace\big[\MM_2^{(k)}(\widetilde{T}\Lambda,\widetilde{T}\Lambda,\Lambda)\big].
\end{equation*}
which justifies our claim. In what follows, we will still use $A,C,T,R$ for $\widetilde{A},\widetilde{C},\widetilde{T},\widetilde{R}$. 

We now prove \eqref{eqt:TraceIneq} with $B=\Lambda$ being diagonal whose diagonal entries are $\lambda_1,\lambda_2,\cdots,\lambda_n\geq0$. Using product properties \eqref{eqt:prodproperty} and identities in \Cref{lem:impidentities}, we rewrite the quantity
\begin{align*}
\mathcal{I}\triangleq &\ \int_0^1ds\trace\big[\MM_1^{(k)}(T\Lambda^s;\Lambda^s)\MM_1^{(k)}(T\Lambda^{1-s};\Lambda^{1-s})\big]-\trace\big[\MM_1^{(k)}(R\Lambda;\Lambda)\big]\\
=&\ \int_0^1ds\left\{\trace\big[\MM_1^{(k)}(T\Lambda^sT\Lambda^{1-s};\Lambda)\big] + \trace\big[\MM_2^{(k)}(T\Lambda^s\Lambda^{1-s},\Lambda^sT\Lambda^{1-s};\Lambda) \big]\right\} -\trace\big[\MM_1^{(k)}(R\Lambda;\Lambda)\big]\\
=&\ \int_0^1ds\left\{\sum_{i=1}^n\big(\sum_{j=1}^nT_{ij}\lambda_j^sT_{ji}\lambda_i^{1-s}\big)d_i^{(n,k)} + \sum_{1\leq i,j\leq n}(T_{ii}\lambda_i\lambda_j^sT_{jj}\lambda_j^{1-s}-T_{ji}\lambda_i\lambda_i^sT_{ij}\lambda_j^{1-s})g_{ij}^{(n,k)}\right\}\\
&\ - \sum_{i=1}^nR_{ii}\lambda_id_i^{(n,k)}.
\end{align*}
Then replacing $b_i$ by $\lambda_id_i^{(n,k)}$ in \eqref{eqt:LiebThm2}, we have by \Cref{lem:LiebThm}
\[\sum_{i=1}^nR_{ii}\lambda_id_i^{(n,k)}\geq \int_0^1ds\sum_{i=1}^n\sum_{j=1}^nT_{ij}(\lambda_jd_j^{(n,k)})^sT_{ji}(\lambda_id_i^{(n,k)})^{1-s}.\]
Therefore we have
\begin{align*}
\mathcal{I}\leq &\ \int_0^1ds\Big\{\sum_{1\leq i,j\leq n}T_{ij}T_{ji}\lambda_j^s\lambda_i^{1-s}d_i^{(n,k)} + \sum_{1\leq i,j\leq n}(T_{ii}T_{jj}\lambda_i\lambda_j-T_{ji}T_{ij}\lambda_i^{1+s}\lambda_j^{1-s})g_{ij}^{(n,k)}\\
&\ \qquad\qquad- \sum_{1\leq i,j\leq n}T_{ij}T_{ji}(\lambda_jd_j^{(n,k)})^s(\lambda_id_i^{(n,k)})^{1-s}\Big\}.
\end{align*}
We now investigate the integrant for any $s\in[0,1]$. We have
\begin{align*}
&\ \sum_{1\leq i,j\leq n}T_{ij}T_{ji}\lambda_j^s\lambda_i^{1-s}d_i^{(n,k)} + \sum_{1\leq i,j\leq n}(T_{ii}T_{jj}\lambda_i\lambda_j-T_{ji}T_{ij}\lambda_i^{1+s}\lambda_j^{1-s})g_{ij}^{(n,k)}\\
&\ -\sum_{1\leq i,j\leq n}T_{ij}T_{ji}(\lambda_jd_j^{(n,k)})^s(\lambda_id_i^{(n,k)})^{1-s}\\
=&\ \sum_{i=1}^nT_{ii}\lambda_id_i^{(n,k)} + \sum_{1\leq i<j\leq n}|T_{ij}|^2(\lambda_j^s\lambda_i^{1-s}d_i^{(n,k)}+\lambda_i^s\lambda_j^{1-s}d_j^{(n,k)})\\
&\ + \sum_{1\leq i,j\leq n}T_{ii}T_{jj}\lambda_i\lambda_jg_{ij}^{(n,k)} - \sum_{1\leq i<j\leq n}|T_{ij}|^2(\lambda_i^{1+s}\lambda_j^{1-s}+\lambda_j^{1+s}\lambda_i^{1-s})g_{ij}^{(n,k)}\\
&\ - \sum_{i=1}^nT_{ii}\lambda_id_i^{(n,k)} - \sum_{1\leq i<j\leq n}|T_{ij}|^2(\lambda_j^s\lambda_i^{1-s}(d_j^{(n,k)})^s(d_i^{(n,k)})^{1-s}+\lambda_i^s\lambda_j^{1-s}(d_i^{(n,k)})^s(d_j^{(n,k)})^{1-s})\\
=&\ \sum_{1\leq i,j\leq n}T_{ii}T_{jj}\lambda_i\lambda_jg_{ij}^{(n,k)}\\
&\ + \sum_{1\leq i<j\leq n}|T_{ij}|^2\Big\{\lambda_j^s\lambda_i^{1-s}d_i^{(n,k)}+\lambda_i^s\lambda_j^{1-s}d_j^{(n,k)}-(\lambda_i^{1+s}\lambda_j^{1-s}+\lambda_j^{1+s}\lambda_i^{1-s})g_{ij}^{(n,k)} \\
&\ \qquad\qquad\qquad\qquad - \lambda_j^s\lambda_i^{1-s}(d_j^{(n,k)})^s(d_i^{(n,k)})^{1-s} - \lambda_i^s\lambda_j^{1-s}(d_i^{(n,k)})^s(d_j^{(n,k)})^{1-s}\Big\}\\
\leq&\ \sum_{1\leq i,j\leq n}T_{ii}T_{jj}\lambda_i\lambda_jg_{ij}^{(n,k)} - 2\sum_{1\leq i<j\leq n}|T_{ij}|^2\lambda_i\lambda_jg_{ij}^{(n,k)}.
\end{align*}
We have used $g_{ij}^{(n,k)}=g_{ji}^{(n,k)}$. The proof of the last inequality above is as follows. For any $s\in [0,1]$, we have a Holder-type inequality for scalars: 
\[(a+b)^s(c+d)^{1-s}\geq a^sc^{1-s}+b^sd^{1-s},\quad a,b,c,d\geq 0.\]
Then using the expansion relations \eqref{eqt:expansion},
\[d_i^{(n,k)}=\lambda_jg_{ij}^{(n,k)}+g_{ij}^{(n,k+1)},\qquad d_j^{(n,k)}=\lambda_ig_{ij}^{(n,k)}+g_{ij}^{(n,k+1)},\]
we have
\begin{align*}
&\ \lambda_j^s\lambda_i^{1-s}d_i^{(n,k)}+\lambda_i^s\lambda_j^{1-s}d_j^{(n,k)} - (\lambda_i^{1+s}\lambda_j^{1-s}+\lambda_j^{1+s}\lambda_i^{1-s})g_{ij}^{(n,k)}\\
&\  - \lambda_j^s\lambda_i^{1-s}(d_j^{(n,k)})^s(d_i^{(n,k)})^{1-s} - \lambda_i^s\lambda_j^{1-s}(d_i^{(n,k)})^s(d_j^{(n,k)})^{1-s}\\
\leq&\ \lambda_j^s\lambda_i^{1-s}(\lambda_jg_{ij}^{(n,k)}+g_{ij}^{(n,k+1)})+\lambda_i^s\lambda_j^{1-s}(\lambda_ig_{ij}^{(n,k)}+g_{ij}^{(n,k+1)}) - (\lambda_i^{1+s}\lambda_j^{1-s}+\lambda_j^{1+s}\lambda_i^{1-s})g_{ij}^{(n,k)}\\
&\ - \lambda_j^s\lambda_i^{1-s}(\lambda_i^s\lambda_j^{1-s}g_{ij}^{(n,k)}+g_{ij}^{(n,k+1)}) - \lambda_i^s\lambda_j^{1-s}(\lambda_j^s\lambda_i^{1-s}g_{ij}^{(n,k)}+g_{ij}^{(n,k+1)})\\
=&\ -2\lambda_i\lambda_jg_{ij}^{(n,k)}.
\end{align*}
Finally using \Cref{lem:impidentities} again, we have 
\begin{align*}
\mathcal{I} \leq&\ \int_0^1ds\left\{\sum_{1\leq i,j\leq n}T_{ii}T_{jj}\lambda_i\lambda_jg_{ij}^{(n,k)} - 2\sum_{1\leq i<j\leq n}|T_{ij}|^2\lambda_i\lambda_jg_{ij}^{(n,k)}\right\}\\
=&\ \sum_{1\leq i,j\leq n}(T_{ii}T_{jj}\lambda_i\lambda_j-T_{ij}T_{ji}\lambda_i\lambda_j)g_{ij}^{(n,k)}. \\
=&\ \trace\big[\MM_2^{(k)}(T\Lambda,T\Lambda,\Lambda)\big].
\end{align*}

\end{proof}

We are now ready to prove \Cref{thm:GeneralizedLiebThm} with all established results.

\begin{proof}[\rm\textbf{Proof of \Cref{thm:GeneralizedLiebThm}}] We first prove the concavity of functions $f_{H,k}(A) = \big(\trace_k \big[\exp\big(H+\ln A\big)\big]\big)^\frac{1}{k}$.
Notice that given any $A\in \mathcal{H}_n^{++}$ and any $C\in \mathcal{H}_n$, there exist some $\epsilon$ such that $A+tC\in\mathcal{H}_n^{++}$ for $t\in(-\epsilon,\epsilon)$, and $f_{H,k}(A+tC)$ is continuously differentiable with respect to $t$ on $(-\epsilon,\epsilon)$. In what follows, any function of $t$ is always assumed to be defined on a reasonable neighborhood of $0$ (so that $A+tC\in \mathcal{H}_n^{++}$). 

Then the concavity of $f_{H,k}(A)$ on $\mathcal{H}_n^{++}$ is equivalently to the statement that $\frac{\partial^2}{\partial t^2}f_{H,k}(A+tC)\leq 0|_{t=0}$ for all choices of $A\in\mathcal{H}_n^{++},C\in\mathcal{H}_n$. Now fix a pair $A,C$, define $B(t) = \exp\big(H+\ln (A+tC)\big)\in \mathcal{H}_n^{++}$ and $g(t) = \trace_k \big[\exp\big(H+\ln (A+tC)\big)\big] = \trace\big[\MM_0^{(k)}(B(t))\big]>0$. Since $f_{H,k}(A+tC)=g(t)^\frac{1}{k}$, and 
\[\frac{\partial^2}{\partial t^2}f_{H,k}(A+tC)=\frac{1}{k}g(t)^{\frac{1}{k}-2}\big(g''(t)g(t)-\frac{k-1}{k}(g'(t))^2\big),\]
we then need to show that $g(0)g''(0)\leq\frac{k-1}{k}(g'(0))^2$. Using the derivative formulas \eqref{eqt:inversederivative} and \eqref{eqt:logderivative}, we have  
\[\frac{\partial}{\partial t}\ln (A+tC) = \int_0^\infty(A+tC+xI_n)^{-1}C(A+tC+xI_n)^{-1}\triangleq T(t),\]
\[\frac{\partial}{\partial t}T(t) = -2\int_0^\infty(A+tC+xI_n)^{-1}C(A+tC+xI_n)^{-1}C(A+tC+xI_n)^{-1} \triangleq -R(t).\] 
Then using formula \eqref{eqt:expderivative}, we can compute the first derivative
\begin{align*}
g'(t) =&\ \frac{\partial}{\partial t}\trace\big[\MM_0^{(k)}(B(t))\big]\\
=&\ \trace \big[\MM_1^{(k)}(B'(t);B(t))\big]\\
=&\ \trace \big[\MM_1^{(k)}\big(\int_0^1dsB(t)^sT(t)B(t)^{1-s};B(t)\big)\big]\\
=&\ \int_0^1ds\trace \big[\MM_1^{(k)}\big(B(t)^sT(t)B(t)^{1-s};B(t)^sB(t)^{1-s}\big)\big]\\
=&\ \int_0^1ds\trace\big[\MM_0^{(k)}(B(t)^s)\MM_1^{(k)}(T(t);I_n)\MM_0^{(k)}(B(t)^{1-s})\big]\\
=&\ \int_0^1ds\trace\big[\MM_1^{(k)}(T(t);I_n)\MM_0^{(k)}(B(t)^{1-s})\MM_0^{(k)}(B(t)^s)\big]\\
=&\ \trace\big[\MM_1^{(k)}(T(t);I_n)\MM_0^{(k)}(B(t))\big].
\end{align*}
We have used the fact that $\MM_1^{(k)}(X;Y)$ is linear in $X$, and so we can pull out the integral symbol. Then the second derivative is
\begin{align*}
g''(t) =&\ \frac{\partial}{\partial t}\trace\big[\MM_1^{(k)}(T(t);I_n)\MM_0^{(k)}(B(t))\big] \\
=&\ \trace\big[\MM_1^{(k)}(T(t);I_n)\MM_1^{(k)}(B'(t);B(t))\big] + \trace\big[\MM_1^{(k)}(T'(t);I_n)\MM_0^{(k)}(B(t))\big]  \\
=&\ \int_0^1ds\trace\big[\MM_1^{(k)}(T(t);I_n)\MM_0^{(k)}(B(t)^s)\MM_1^{(k)}(T(t);I_n)\MM_0^{(k)}(B(t)^{1-s})\big] \\
&\ - \trace\big[\MM_1^{(k)}(R(t);I_n)\MM_0^{(k)}(B(t))\big].
\end{align*}
Write $T=T(0)$, $R=R(0)$ and $B=B(0)$. Then using definitions \eqref{eqt:threeops}, identity \eqref{eqt:extrace2mixdiscriminant}, the Alexandrov-Fenchel inequality (\Cref{thm:AFinequality}) and \Cref{lem:TraceIneq}, we have
\begin{align*}
&\ g(0)g''(0) \\
=&\ \trace\big[\MM_0^{(k)}(B)\big]\Big\{\int_0^1ds\trace\big[\MM_1^{(k)}(T;I_n)\MM_0^{(k)}(B^s)\MM_1^{(k)}(T;I_n)\MM_0^{(k)}(B^{1-s})\big] \\
&\ \qquad\qquad\qquad\quad - \trace\big[\MM_1^{(k)}(R;I_n)\MM_0^{(k)}(B)\big]\Big\}\\
=&\ \trace\big[\MM_0^{(k)}(B)\big]\Big\{\int_0^1ds\trace\big[\MM_1^{(k)}(TB^s;B^s)\MM_1^{(k)}(TB^{1-s};B^{1-s})\big] \\
&\ \qquad\qquad\qquad\qquad\qquad - \trace\big[\MM_1^{(k)}(RB;B)\big]\Big\} \\
\leq&\ \trace\big[\MM_0^{(k)}(B)\big]\trace\big[\MM_2^{(k)}(TB,TB,B)] \\
=&\ \frac{n!}{k!(n-k)!}D(\underbrace{B,\cdots,B}_{k},\underbrace{I_n,\cdots,I_n}_{n-k})\cdot \frac{n!}{(k-2)!(n-k)!}D(TB,TB,\underbrace{B\cdots,B}_{k-2},\underbrace{I_n,\cdots,I_n}_{n-k})\\
\leq& \frac{k-1}{k}\Big(\frac{n!}{(k-1)!(n-k)!}D(TB,\underbrace{B\cdots,B}_{k-1},\underbrace{I_n,\cdots,I_n}_{n-k})\Big)^2 \\
=&\ \frac{k-1}{k}\trace\big[\MM_1^{(k)}(TB,B)\big]^2. \\
=&\ \frac{k-1}{k}(g'(0))^2.
\end{align*}
The concavity of $f_{H,k}(A)$ on $\mathcal{H}_n^{++}$ then follows. 

Next we prove the equivalence of (i) the concavity of the functions $f_{H,k}(A)$ on $\mathcal{H}_n^{++}$ and (ii) the concavity of the functions $\tilde{f}_{H,k}= \ln\trace_k \big[\exp\big(H+\ln A\big)\big]$ on $\mathcal{H}_n^{++}$. (i) $\Rightarrow$ (ii) is trivial. To prove (ii) $\Rightarrow$ (i), we need the following lemma.

\begin{lemma}
\label{lem:homogeneous}
Let function $f:(0,+\infty)^n\rightarrow(0,+\infty)$ be homogeneous of order $s>0$, i.e. $f(\lambda x) = \lambda^sf(x)$, for $\lambda> 0$. Then $f(x)^\frac{1}{s}$ is concave if and only if $\ln f(x)$ is concave.
\end{lemma}
\begin{proof}
One direction is trivial. If $f(x)^\frac{1}{s}$ is concave, then $\ln f(x) = s\ln(f(x)^\frac{1}{s})$ is concave since $\ln(\cdot)$ is monotone and concave on $(0,+\infty)$.

Conversely, if $\ln f(x)$ is concave, then $f(\tau x+(1-\tau)y)\geq f(x)^\tau f(y)^{1-\tau}$, for any $x,y\in (0,+\infty),\tau\in[0,1]$. Now for any fixed $x,y\in (0,+\infty),\tau\in[0,1]$, we define $M = \tau f(x)^\frac{1}{s}+(1-\tau)f(y)^\frac{1}{s}$. We then have
\begin{align*}
f(\tau x+(1-\tau)y)^\frac{1}{s} =&\ f\left(\frac{\tau f(x)^\frac{1}{s}}{M}\frac{Mx}{f(x)^\frac{1}{s}}+\frac{(1-\tau)f(y)^\frac{1}{s}}{M}\frac{My}{f(y)^\frac{1}{s}}\right)^\frac{1}{s}\\
\geq&\  f\left(\frac{Mx}{f(x)^\frac{1}{s}}\right)^{\frac{\tau f(x)^\frac{1}{s}}{M}\cdot\frac{1}{s}}f\left(\frac{My}{f(y)^\frac{1}{s}}\right)^{\frac{(1-\tau)f(y)^\frac{1}{s}}{M}\cdot\frac{1}{s}}\\
=&\ (M^s)^{\frac{\tau f(x)^\frac{1}{s}}{M}\cdot\frac{1}{s}+\frac{(1-\tau)f(y)^\frac{1}{s}}{M}\cdot\frac{1}{s}} \\
=&\ M.
\end{align*}
Therefore $f(x)^\frac{1}{s}$ is concave.
\end{proof}

Let $x=(x_1,x_2)\in(0,+\infty)^2$. Define $f(x) = \trace_k \big[\exp\big(H+\ln (x_1A_1+x_2A_2)\big)\big]$. One can easily verify that $f_{H,k}(A)$ being concave on $\mathcal{H}_n^{++}$ is equivalent to $f(x)^\frac{1}{k}$ being concave on $(0,+\infty)^2$ for arbitrary but fixed choice of $A_1,A_2\in\mathcal{H}_n^{++},H\in\mathcal{H}_n$. Similarly, $\tilde{f}_{H,k}(A)$ being concave on $\mathcal{H}_n^{++}$ is equivalent to $\ln f(x)$ being concave on $(0,+\infty)^2$ for arbitrary but fixed choice of $A_1,A_2\in\mathcal{H}_n^{++},H\in\mathcal{H}_n$. Using the definition of the $k$-trace $\trace_k$, it is easy to check that $f(x)$ is homogeneous of order $k$. By \Cref{lem:homogeneous}, we know $f(x)^\frac{1}{k}$ is concave if and only if $\ln f(x)$ is concave. Therefore we have (i) $\Leftrightarrow$ (ii). 
\end{proof}

We close this section with a conjecture. Epstein \cite{epstein1973remarks} proved that the function 
\begin{equation}
A\longmapsto \trace\big[(B^*A^pB)^\frac{1}{p}\big]
\end{equation}
is concave on $\mathcal{H}_n^+$ for any $0\leq q\leq 1$ and any $B\in\mathbb{C}^{n\times n}$. If we choose $p=\frac{1}{m}$, $B = \exp(\frac{1}{2m}H)$ for some $H\in\mathcal{H}_n$, then taking $m\rightarrow+\infty$ and using the Lie product formula ($\lim_{m\rightarrow+\infty}(\exp(\frac{1}{m}X)\exp(\frac{1}{m}Y))^m=\exp(X+Y)$), we immediately obtain the Lieb's theorem that $A\mapsto \trace\big[\exp(H+\ln A)\big]$ is concave on $\mathcal{H}_n^{++}$. In fact, Carlen \cite{carlen2010trace} showed that Epstein's result, and its generalized version, can be derived from the Lieb's concavity theorem (Theorem 1 \cite{LIEB1973267}). But we know that the Lieb's theorem is equivalent to the Lieb's concavity theorem (see \cite{LIEB1973267}), and hence is equivalent to Epstein's result. Therefore, it is reasonable to make the conjecture that the function 
\begin{equation}
A\longmapsto \big(\trace_k\big[(B^*A^pB)^\frac{1}{p}\big]\big)^\frac{1}{k}
\end{equation}
is concave on $\mathcal{H}_n^+$. We shall discuss this in future works.

\section{Application to a sum of random matrices}
\label{sec:Application}
In many problems, the assemble of a large complicated matrix is by sampling independent random matrices with simpler structures. To study the spectrum of the expected matrix by only evaluating the spectrum of the sample mean, we need to know how the latter deviates from the former. Therefore, we often need to estimate the spectrum of random matrices of the form $Y=\sum_{i=1}^mX^{(i)}$, where $X^{(i)},1\leq i\leq m$ are independent random matrices of the same size. In particular, we consider the Hermitian case where $X^{(i)}\in\mathcal{H}_n$. An important tool to study the extreme eigenvalues of a sum of random matrices is the following master bounds by Tropp (Theorem 3.6.1 \cite{MAL-048}). Consider a finite sequence of independent, random matrices $\{X^{(i)}\}_{i=1}^m\subset\mathcal{H}_n$. Then
\begin{subequations}
\begin{align}
&\mathbb{E}\lambda_{\max}\Big(\sum_{i=1}^mX^{(i)}\Big) \leq \inf_{\theta>0}\ \frac{1}{\theta}\ln \trace\big[\exp\big(\sum_{i=1}^m\ln\mathbb{E}\exp(\theta X^{(i)})\big)\big],\\
&\mathbb{E}\lambda_{\min}\Big(\sum_{i=1}^mX^{(i)}\Big) \geq \sup_{\theta<0}\ \frac{1}{\theta}\ln \trace\big[\exp\big(\sum_{i=1}^m\ln\mathbb{E}\exp(\theta X^{(i)})\big)\big].
\end{align}
\label{eqt:TroppMasterExp}
\end{subequations}
Furthermore, for all $t\in \mathbb{R}$,
\begin{subequations}
\begin{align}
&\mathbb{P}\left\{\lambda_{\max}\Big(\sum_{i=1}^mX^{(i)}\Big) \geq  t\right\} \leq \inf_{\theta>0}\ e^{-\theta t}\trace\big[\exp\big(\sum_{i=1}^m\ln\mathbb{E}\exp(\theta X^{(i)})\big)\big],\\
& \mathbb{P}\left\{\lambda_{\min}\Big(\sum_{i=1}^mX^{(i)}\Big) \leq t\right\}\leq \inf_{\theta<0}\ e^{-\theta t}\trace\big[\exp\big(\sum_{i=1}^m\ln\mathbb{E}\exp(\theta X^{(i)})\big)\big].
\end{align}
\label{eqt:TroppMasterTail}
\end{subequations}
Tropp's proof of the master bounds rely on a critical use of the Lieb's theorem. To be specific, Tropp used the concavity of $A\mapsto \trace\big[\exp(H+\ln A)\big]$ to prove the subadditivity of matrix cumulant generating function. For the sequence $\{X^{(i)}\}_{i=1}^m\subset\mathcal{H}_n$ under the same setting, 
\begin{equation}
\mathbb{E}\trace\big[\exp\big(\sum_{i=1}^mX^{(i)}\big)\big]\leq \trace\big[\exp\big(\sum_{i=1}^m\ln\mathbb{E}\exp X^{(i)}\big)\big].
\label{eqt:subadditivity}
\end{equation}

Similarly, to prove \Cref{thm:ksummasterbound}, we need to extend \eqref{eqt:subadditivity} to the following lemma using our generalized Lieb's theorem.

\begin{lemma}
\label{lem:MultiConcave}
Let $A^{(1)},A^{(2)},\cdots,A^{(m)}\in \mathcal{H}_n^{++}$ be $m$ independent, random, positive definite matrices. Then we have for any $1\leq k\leq n$,
\begin{subequations}
\begin{align}
&\mathbb{E}\big(\trace_k\big[\exp(\sum_{i=1}^m\ln A^{(i)})\big]\big)^\frac{1}{k}\leq \big(\trace_k\big[\exp(\sum_{i=1}^m\ln \mathbb{E} A^{(i)})\big]\big)^\frac{1}{k},
\label{eqt:multiconcave1}\\
&\mathbb{E}\ln\trace_k\big[\exp(\sum_{i=1}^m\ln A^{(i)})\big]\leq \ln\trace_k\big[\exp(\sum_{i=1}^m\ln \mathbb{E} A^{(i)})\big].
\label{eqt:multiconcave2}
\end{align}
\end{subequations}
\end{lemma}

\begin{proof}
We here only prove \eqref{eqt:multiconcave1}. The proof for \eqref{eqt:multiconcave2} is similar. Since each random matrix $A^{(i)}$ always lies in $\mathcal{H}_n^{++}$, we may apply \Cref{thm:GeneralizedLiebThm} to get a Jensen's inequality 
\[\mathbb{E}\big(\trace_k\big[\exp(H+\ln A^{(i)})\big]\big)^\frac{1}{k}\leq \big(\trace_k\big[\exp(H+\ln \mathbb{E} A^{(i)})\big]\big)^\frac{1}{k},\]
for arbitrary $H\in \mathcal{H}_n$. And since $A^{(1)},A^{(2)},\cdots,A^{(m)}$ are independent, we can split the expectation $\mathbb{E}$ into $\mathbb{E}_1\mathbb{E}_2\cdots\mathbb{E}_m$, where $\mathbb{E}_i$ is the expectation operator with respect to the random matrix $A^{(i)}$. So then we may apply \Cref{thm:GeneralizedLiebThm} repeatedly to get 
\begin{align*}
&\ \mathbb{E}\big(\trace_k\big[\exp(\sum_{i=1}^m\ln A^{(i)})\big]\big)^\frac{1}{k}\\
=&\ \mathbb{E}_1\cdots\mathbb{E}_m \big(\trace_k\big[\exp(\sum_{i=1}^{m-1}\ln A^{(i)}+\ln A^{(m)})\big]\big)^\frac{1}{k}\\
\leq&\ \mathbb{E}_1\cdots\mathbb{E}_{m-1} \big(\trace_k\big[\exp(\sum_{i=1}^{m-1}\ln A^{(i)}+\ln \mathbb{E}_mA^{(m)})\big]\big)^\frac{1}{k}\\
&\ \cdots\\
\leq&\ \big(\trace_k\big[\exp(\sum_{i=1}^{m}\ln \mathbb{E}A^{(i)})\big]\big)^\frac{1}{k}.
\end{align*}
\end{proof}

Our proof of \Cref{thm:ksummasterbound} will basically follow Tropp's proof of Theorem 3.6.1 in \cite{MAL-048}, but with the normal trace $\trace$ replaced by the general $k$-trace $\trace_k$ for $1\leq k\leq n$.

\begin{proof}[\rm\textbf{Proof of \Cref{thm:ksummasterbound}}]
We first prove \eqref{eqt:masterexpect1} and \eqref{eqt:mastertail1}. The first inequality in \eqref{eqt:masterexpect1} is trivial, because the sum of a Hermitian matrix's $k$ largest eigenvalues is a convex function of the matrix itself. Indeed we have
\[\mathbb{E}\sum_{i=1}^k\lambda_i(Y) = \mathbb{E}\ \sup_{\begin{subarray}{c}Q\in \mathbb{C}^{n\times k}\\
Q^*Q=I_k\end{subarray}} \trace[Q^*YQ] \geq  \sup_{\begin{subarray}{c}Q\in \mathbb{C}^{n\times k}\\
Q^*Q=I_k\end{subarray}} \trace[Q^*(\mathbb{E}Y)Q] = \sum_{i=1}^k\lambda_i(\mathbb{E}Y).\]
For the second inequality in \eqref{eqt:masterexpect1}, we apply a similar technique, the matrix Laplace transform, as in \cite{MAL-048}. The difference is that in the first step, we don't switch the expectation operator and the logarithm. For any $\theta>0$, we have
\begin{equation*}
\mathbb{E}\sum_{i=1}^k\lambda_i(Y)= \frac{1}{\theta}\mathbb{E}\ln\exp\big(\sum_{i=1}^k\lambda_i(\theta Y)\big)=\frac{1}{\theta}\mathbb{E}\ln \big(\prod_{i=1}^k\lambda_i\big(\exp (\theta Y)\big)\big) \leq \frac{1}{\theta}\mathbb{E}\ln \trace_k\big[\exp (\theta Y)\big].
\end{equation*}
We have used the fact that $\prod_{i=1}^k\lambda_i(A)\leq\trace_k\big[A\big]$ for any $A\in\mathcal{H}_n^+$. Next we define the random matrices $A^{(i)} = \exp (\theta X^{(i)})\in\mathcal{H}_n^{++},1\leq i\leq m$. Since $X^{(i)},1\leq i\leq m$, are independent, $A^{(i)},1\leq i\leq m$ are also independent. Therefore we may apply inequality \eqref{eqt:multiconcave2} in \Cref{lem:MultiConcave} to get 
\begin{align*}
\mathbb{E}\ln \trace_k\big[\exp (\theta Y)\big]=&\ \mathbb{E}\ln \trace_k\big[\exp \big(\sum_{i=1}^m\theta X^{(i)}\big)\big] = \mathbb{E}\ln \trace_k\big[\exp \big(\sum_{i=1}^m\ln A^{(i)}\big)\big]\\
\leq&\ \ln \trace_k\big[\exp \big(\sum_{i=1}^m\ln \mathbb{E}A^{(i)}\big)\big] = \ln \trace_k\big[\exp \big(\sum_{i=1}^m\ln \mathbb{E}\exp(\theta X^{(i)})\big)\big].
\end{align*}
Since $\theta>0$ is arbitrary, we thus have 
\[\mathbb{E}\sum_{i=1}^k\lambda_i(Y)\leq \inf_{\theta>0}\ \frac{1}{\theta} \ln \trace_k\big[\exp \big(\sum_{i=1}^m\ln \mathbb{E}\exp(\theta X^{(i)})\big)\big]. \]
The proof of \eqref{eqt:mastertail1} shares a similar spirit, except that we use \eqref{eqt:multiconcave1} instead of \eqref{eqt:multiconcave2}. For any $t\in \mathbb{R},\theta>0$, we use the Markov's inequality to obtain
\begin{align*}
\mathbb{P}\left\{\sum_{i=1}^k\lambda_i(Y)\geq t\right\} =&\ \mathbb{P}\left\{\exp\Big(\frac{\theta}{k}\sum_{i=1}^k\lambda_i(Y)\Big)\geq e^\frac{\theta t}{k}\right\}\\
\leq&\ e^{-\frac{\theta t}{k}} \mathbb{E}\exp\Big(\frac{1}{k}\sum_{i=1}^k\lambda_i(\theta Y)\Big) \leq e^{-\frac{\theta t}{k}} \mathbb{E}\Big[\big(\trace_k\exp(\theta Y)\big)^\frac{1}{k}\Big].
\end{align*} 
Then again by defining $A^{(i)} = \exp (\theta X^{(i)})\in\mathcal{H}_n^{++},1\leq i\leq m$, we may apply inequality \eqref{eqt:multiconcave1} in \Cref{lem:MultiConcave} to obtain 
\begin{align*}
\mathbb{E}\Big[\big(\trace_k\exp(\theta Y)\big)^\frac{1}{k}\Big] \leq&\ \Big(\trace_k\exp\big(\sum_{i=1}^m\ln\mathbb{E}\exp(\theta X^{(i)})\big)\Big)^\frac{1}{k}
\end{align*}
Since $\theta>0$ is arbitrary, we thus have
\[\mathbb{P}\left\{\sum_{i=1}^k\lambda_i(Y)\geq t\right\}\leq \inf_{\theta>0}\  e^{-\frac{\theta t}{k}}\Big(\trace_k\exp\big(\sum_{i=1}^m\ln\mathbb{E}\exp(\theta X^{(i)})\big)\Big)^\frac{1}{k}.\]

We proceed to \eqref{eqt:masterexpect2} and \eqref{eqt:mastertail2}. The first inequality \eqref{eqt:masterexpect2} can be similarly verified by noticing that the sum of a Hermitian matrix's $k$ smallest eigenvalues is a concave function of the matrix itself. For the second inequality in \eqref{eqt:masterexpect2} and inequality \eqref{eqt:mastertail2}, we only need to consider arbitrary $\theta<0$, and use the fact that $\theta\lambda_{n-i+1}(A)=\lambda_i(\theta A)$ for any $A\in \mathcal{H}_n$. Then repeating the arguments for \eqref{eqt:masterexpect1} and \eqref{eqt:mastertail1}, we can similarly show that 
\[\mathbb{E}\sum_{i=1}^k\lambda_{n-i+1}(Y)\leq \sup_{\theta<0}\ \frac{1}{\theta} \ln \trace_k\big[\exp \big(\sum_{i=1}^m\ln \mathbb{E}\exp(\theta X^{(i)})\big)\big], \]
and 
\[\mathbb{P}\left\{\sum_{i=1}^k\lambda_{n-i+1}(Y)\leq t\right\}\leq \inf_{\theta<0}\  e^{-\frac{\theta t}{k}}\Big(\trace_k\exp\big(\sum_{i=1}^m\ln\mathbb{E}\exp(\theta X^{(i)})\big)\Big)^\frac{1}{k}.\]
\end{proof}

We next consider a more specific case of random matrices taking the form $Y=\sum_{i=1}^mX^{(i)}$. In particular, we assume that each $X^{(i)}\in\mathcal{H}_n$ satisfies $0\leq\lambda_{\min}(X^{(i)})\leq \lambda_{\max}(X^{(i)})\leq c$ for some constant $c\geq0$. One of the most studied scenarios in this setting arises with an undirected, no-selfloop, randomly weighted graph $G=(V,E,W)$ of $n$ vertices. All the weights $w_{ij}$ for all edges $e_{ij},i<j$ are uniformly bounded and follow independent distributions. Then the Laplacian of such random graph is given by $L=\sum_{1\leq i<j\leq n}w_{ij}X^{(i,j)}$, where  
\[X^{(i,j)}=
\begin{array}{cc}
 &  \begin{array}{ccccc}
 	& i & & j & 
 	\end{array}\\
 	\begin{array}{c}
 	  i \\ \\ j \\ 
 	\end{array} & 
 	\left[\begin{array}{ccccc}
	& & & & \\
	& 1 & & -1 & \\
	& & & & \\
	& -1 & & 1 & \\
	& & & & 
	\end{array}\right]
\end{array},\quad i<j,
\]
is the sub-Laplacian corresponding to the edge $e_{ij}$ with unit weight. In particular, if each weight follows a Bernoulli distribution $B(1,p)$ for some uniform constant $p\in[0,1]$, the random graph is known as the famous Erd\H{o}s-R\'enyi model.

For these kind of problems, one may want to study how the eigenvalues of $Y=\sum_{i=1}^mX^{(i)}$ deviate from the corresponding eigenvalues of $\mathbb{E}Y$. For such purposes, Tropp \cite{MAL-048} used the master bounds \eqref{eqt:TroppMasterTail} and \eqref{eqt:TroppMasterExp}, and delicate bounds for the matrix moment generating function (Lemma 5.4.1 in \cite{MAL-048}) to prove the following expectation estimates 
\begin{subequations}
\label{eqt:TroopChernoffExp}
\begin{align}
&\lambda_{\max}(\mathbb{E}Y)\leq \mathbb{E}\lambda_{\max}(Y)\leq \inf_{\theta>0}\ \frac{e^\theta-1}{\theta}\lambda_{\max}(\mathbb{E}Y)+\frac{c}{\theta}\ln n,\\
&\lambda_{\min}(\mathbb{E}Y)\geq \mathbb{E}\lambda_{\min}(Y) \geq \sup_{\theta>0}\ \frac{1-e^{-\theta}}{\theta}\lambda_{\min}(\mathbb{E}Y)- \frac{c}{\theta}\ln n,
\end{align}
\end{subequations}
and Chernoff-type tail bounds 
\begin{subequations}
\label{eqt:TroopChernoffTail}
\begin{align}
&\mathbb{P}\left\{\lambda_{\max}(Y)\geq (1+\varepsilon)\lambda_{\max}(\mathbb{E}Y)\right\} \leq n\left(\frac{e^\varepsilon}{(1+\varepsilon)^{1+\varepsilon}}\right)^{\lambda_{\max}(\mathbb{E}Y)/c},\quad \varepsilon\geq 0,\\
&\mathbb{P}\left\{\lambda_{\min}(Y)\leq (1-\varepsilon)\lambda_{\min}(\mathbb{E}Y)\right\} \leq n\left(\frac{e^{-\varepsilon}}{(1-\varepsilon)^{1-\varepsilon}}\right)^{\lambda_{\min}(\mathbb{E}Y)/c},\quad \varepsilon\in[0,1),
\end{align}
\end{subequations}
for the largest and the smallest eigenvalues of $Y$ and $\mathbb{E}Y$. With \Cref{thm:ksummasterbound}, we shall extend Tropp's results to the following analog theorem.

\begin{thm} 
\label{thm:ksumChernoff}
Given any finite sequence of independent, random matrices $\{X^{(i)}\}_{i=1}^m\subset\mathcal{H}_n$, let $Y=\sum_{i=1}^mX^{(i)}$. Assume that for each $i$, $0\leq \lambda_n(X^{(i)})\leq \lambda_1(X^{(i)})\leq c$ for some uniform constants $c\geq 0$. Then for any $1\leq k\leq n$, we have expectation estimates
\begin{subequations}
\begin{align}
&\mathbb{E}\sum_{i=1}^k\lambda_i(Y) \leq \inf_{\theta>0}\ \frac{e^\theta-1}{\theta}\sum_{i=1}^k\lambda_i(\mathbb{E}Y)+\frac{c}{\theta}\ln\binom{n}{k},
\label{eqt:chernoffexpect1}\\
&\mathbb{E}\sum_{i=1}^k\lambda_{n-i+1}(Y) \geq \sup_{\theta>0}\ \frac{1-e^{-\theta}}{\theta}\sum_{i=1}^k\lambda_{n-i+1}(\mathbb{E}Y)- \frac{c}{\theta}\ln\binom{n}{k},
\label{eqt:chernoffexpect2}
\end{align}
\label{eqt:chernoffexpect}
\end{subequations}
and tail bounds 
\begin{subequations}
\begin{align}
&\mathbb{P}\left\{\sum_{i=1}^k\lambda_i(Y)\right. \geq  \left.(1+\varepsilon)\sum_{i=1}^k\lambda_i(\mathbb{E}Y)\right\}
\label{eqt:chernofftail1}\\
 & \qquad\qquad\qquad \leq \binom{n}{k}^\frac{1}{k}\left(\frac{e^\varepsilon}{(1+\varepsilon)^{1+\varepsilon}}\right)^{\frac{1}{ck}\sum_{i=1}^k\lambda_i(\mathbb{E}Y)}, \quad \varepsilon \geq 0,\nonumber
\\
& \mathbb{P}\left\{\sum_{i=1}^k\lambda_{n-i+1}(Y)\right. \leq \left.(1-\varepsilon)\sum_{i=1}^k\lambda_{n-i+1}(\mathbb{E}Y)\right\}
\label{eqt:chernofftail2}\\
& \qquad\qquad\qquad \leq \binom{n}{k}^\frac{1}{k}\left(\frac{e^{-\varepsilon}}{(1-\varepsilon)^{1-\varepsilon}}\right)^{\frac{1}{ck}\sum_{i=1}^k\lambda_{n-i+1}(\mathbb{E}Y)},\quad \varepsilon\in[0,1).\nonumber
\end{align}
\label{eqt:chernofftail}
\end{subequations}
\end{thm}

\begin{proof}
Since $0\leq \lambda_n(X^{(i)})\leq \lambda_1(X^{(i)})\leq c$, we can use lemma 5.4.1 in \cite{MAL-048} to obtain the estimate 
\[\ln \mathbb{E}\exp(\theta X^{(i)})\leq \frac{e^{\theta c}-1}{c}\mathbb{E}X^{(i)}=g(\theta)\mathbb{E}X^{(i)},\quad \theta\in \mathbb{R},\]
where $g(\theta) = \frac{e^{\theta c}-1}{c}$. So we have 
\begin{align*}
\trace_k\big[\exp \big(\sum_{i=1}^m\ln \mathbb{E}\exp(\theta X^{(i)})\big)\big]\leq&\ \trace_k\big[\exp \big(g(\theta)\sum_{i=1}^m\mathbb{E}X^{(i)}\big)\big] = \trace_k\big[\exp \big(g(\theta)\mathbb{E}Y\big)\big]\\
\leq&\ \binom{n}{k}\prod_{i=1}^k\lambda_i\big(\exp \big(g(\theta)\mathbb{E}Y\big)\big) = \binom{n}{k}\exp\big(\sum_{i=1}^k\lambda_i\big(g(\theta)\mathbb{E}Y\big)\big).
\end{align*}
We have used that fact that $\trace_k\big[A\big]\leq\binom{n}{k}\prod_{i=1}^k\lambda_i(A)$ for any $A\in\mathcal{H}_n^+$. Notice that for $\theta>0$, $g(\theta) = \frac{e^{\theta c}-1}{c}>0$. We then use \eqref{eqt:masterexpect1} in \Cref{thm:ksummasterbound} to get 
\[\mathbb{E}\sum_{i=1}^k\lambda_i(Y)\leq \inf_{\theta>0} \frac{1}{\theta} \ln\left(\binom{n}{k}\exp\big(\sum_{i=1}^k\lambda_i\big(g(\theta)\mathbb{E}Y\big)\big)\right)  =\inf_{\theta>0} \frac{g(\theta)}{\theta}\sum_{i=1}^k\lambda_i(\mathbb{E}Y)+\frac{1}{\theta}\ln\binom{n}{k}.\]
As mentioned in \cite{MAL-048}, this infimum does not admit a closed form. By making change of variable $\theta\rightarrow \theta/c$, we obtain \eqref{eqt:chernoffexpect1}
\[\mathbb{E}\sum_{i=1}^k\lambda_i(Y)\leq \inf_{\theta>0}\frac{e^\theta-1}{\theta}\sum_{i=1}^k\lambda_i(\mathbb{E}Y)+\frac{c}{\theta}\ln\binom{n}{k}.\]
Similarly, We apply \eqref{eqt:mastertail1} in \Cref{thm:ksummasterbound} to get
\begin{align*}\mathbb{P}\left\{\sum_{i=1}^k\lambda_i(Y) \geq  t\right\} \leq&\ \inf_{\theta>0}\ e^{-\frac{\theta t}{k}}\left(\binom{n}{k}\exp\big(\sum_{i=1}^k\lambda_i\big(g(\theta)\mathbb{E}Y\big)\big)\right)^\frac{1}{k}\\
=&\ \inf_{\theta>0}\ e^{-\frac{\theta t}{k}}\binom{n}{k}^\frac{1}{k}\exp\Big(\frac{g(\theta)}{k}\sum_{i=1}^k\lambda_i(\mathbb{E}Y)\Big).
\end{align*}
If we choose $t= (1+\varepsilon)\sum_{i=1}^k\lambda_i(\mathbb{E}Y)$ for $\varepsilon\geq0$, we have 
\begin{equation*}
\mathbb{P}\left\{\sum_{i=1}^k\lambda_i(Y)\geq (1+\varepsilon)\sum_{i=1}^k\lambda_i(\mathbb{E}Y)\right\}\leq \inf_{\theta>0}\ \binom{n}{k}^\frac{1}{k}\exp\Big(\big(g(\theta)-(1+\varepsilon)\theta\big)\frac{1}{k}\sum_{i=1}^k\lambda_i(\mathbb{E}Y)\Big).
\end{equation*}
Minimizing the right hand side with $\theta = \frac{\ln(1+\varepsilon)}{c}$ gives \eqref{eqt:chernofftail1}.

Now consider $\theta<0$, we have $g(\theta) = \frac{e^{\theta c}-1}{c}<0$. We then use \eqref{eqt:masterexpect2} in \Cref{thm:ksummasterbound} to get 
\begin{align*}
\mathbb{E}\sum_{i=1}^k\lambda_{n-i+1}(Y)\geq&\ \sup_{\theta<0} \frac{1}{\theta} \ln\left(\binom{n}{k}\exp\big(\sum_{i=1}^k\lambda_i\big(g(\theta)\mathbb{E}Y\big)\big)\right)\\  
=&\ \sup_{\theta<0} \frac{g(\theta)}{\theta}\sum_{i=1}^k\lambda_{n-i+1}(\mathbb{E}Y)+\frac{1}{\theta}\ln\binom{n}{k}.
\end{align*}
We have used $\lambda_{i}(g(\theta)\mathbb{E}Y)=g(\theta)\lambda_{n-i+1}(\mathbb{E}Y)$ when $g(\theta)<0$. By making change of variable $\theta\rightarrow-\theta/c$, we obtain \eqref{eqt:chernoffexpect2}
\[\mathbb{E}\sum_{i=1}^k\lambda_{n-i+1}(Y)\geq \sup_{\theta>0}\frac{1-e^{-\theta}}{\theta}\sum_{i=1}^k\lambda_{n-i+1}(\mathbb{E}Y)- \frac{c}{\theta}\ln\binom{n}{k}.\]
Similarly, We apply \eqref{eqt:mastertail2} in \Cref{thm:ksummasterbound} to get
\begin{align*}\mathbb{P}\left\{\sum_{i=1}^k\lambda_{n-i+1}(Y) \leq  t\right\} \leq&\ \inf_{\theta<0}\ e^{-\frac{\theta t}{k}}\left(\binom{n}{k}\exp\big(\sum_{i=1}^k\lambda_i\big(g(\theta)\mathbb{E}Y\big)\big)\right)^\frac{1}{k}\\
=&\ \inf_{\theta<0}\ e^{-\frac{\theta t}{k}}\binom{n}{k}^\frac{1}{k}\exp\Big(\frac{g(\theta)}{k}\sum_{i=1}^k\lambda_{n-i+1}(\mathbb{E}Y)\Big).
\end{align*}
If we choose $t= (1-\varepsilon)\sum_{i=1}^k\lambda_{n-i+1}(\mathbb{E}Y)$ for $\varepsilon\in[0,1)$, we have 
\begin{align*}
\mathbb{P}\left\{\sum_{i=1}^k\lambda_{n-i+1}(Y)\right.\leq & \left.(1-\varepsilon)\sum_{i=1}^k\lambda_{n-i+1}(\mathbb{E}Y)\right\}\\
\leq&\ \inf_{\theta<0}\ \binom{n}{k}^\frac{1}{k}\exp\Big(\big(g(\theta)-(1-\varepsilon)\theta\big)\frac{1}{k}\sum_{i=1}^k\lambda_{n-i+1}(\mathbb{E}Y)\Big).
\end{align*}
Minimizing the right hand sids with $\theta=\frac{\ln(1-\varepsilon)}{c}$ gives \eqref{eqt:chernofftail2}.
\end{proof}

In Tropp's results \eqref{eqt:TroopChernoffExp}, namely the case $k=1$, the cost of ``switching'' $\lambda$ and $\mathbb{E}$ is of scale $\ln n$. In our estimates \eqref{eqt:chernoffexpect}, the gap factor becomes $\ln\binom{n}{k}\leq k\ln n$ that grows only sub-linearly in $k$, which is reasonable as we are estimating the sum of the $k$ largest (or smallest) eigenvalues. We shall further compare our estimates to another related work. Tropp et al. \cite{gittens2011tail} introduced a supspace argument based on Courant--Fischer characterization of eigenvalues to prove tail bounds for all eigenvalues of $Y=\sum_{i=1}^mX^{(i)}$. Though not stated in \cite{gittens2011tail}, the following expectation estimates for all eigenvalues can also be established using the supspace argument. Give any finite sequence of independent, random matrices $\{X^{(i)}\}_{i=1}^m$ under the same assumption as in \Cref{thm:ksumChernoff}, and $Y=\sum_{i=1}^mX^{(i)}$, we have for any $1\leq k\leq n$,
\begin{subequations}
\begin{align}
&\mathbb{E}\lambda_k(Y)\leq \inf_{\theta>0}\ \frac{e^\theta-1}{\theta}\lambda_k(\mathbb{E}Y)+\frac{c}{\theta}\ln (n-k+1),\label{eqt:Troppsubspace1}\\
&\mathbb{E}\lambda_k(Y) \geq \sup_{\theta>0}\ \frac{1-e^{-\theta}}{\theta}\lambda_k(\mathbb{E}Y)- \frac{c}{\theta}\ln k.\label{eqt:Troppsubspace2}
\end{align}
\end{subequations}
Summing \eqref{eqt:Troppsubspace1} (or \eqref{eqt:Troppsubspace2}) for the $k$ largest (or smallest) eigenvalues, we immediately obtain 
\begin{subequations}
\begin{align}
&\mathbb{E}\sum_{i=1}^k\lambda_i(Y)\leq \inf_{\theta>0}\ \frac{e^\theta-1}{\theta}\sum_{i=1}^k\lambda_i(\mathbb{E}Y)+\frac{c}{\theta}\ln \prod_{i=1}^k(n-i+1),\\
&\mathbb{E}\sum_{i=1}^k\lambda_{n-i+1}(Y) \geq \sup_{\theta>0}\ \frac{1-e^{-\theta}}{\theta}\sum_{i=1}^k\lambda_{n-i+1}(\mathbb{E}Y)- \frac{c}{\theta}\ln \prod_{i=1}^k(n-i+1).
\end{align}
\end{subequations}
Therefore, our expectation estimates \eqref{eqt:chernoffexpect1} and \eqref{eqt:chernoffexpect2} are sharper for partial sums of eigenvalues, as $\ln \binom{n}{k}<\ln \prod_{i=1}^k(n-i+1)$ for $k>1$. In particular, if one choose $k$ to be a fixed porportion of $n$, then $\ln \binom{n}{k} = O(k)$, while $\ln \prod_{i=1}^k(n-i+1) = O(k\ln n)$. Our results are then better by a factor $\ln n$. 

At last, we remark that if we combine \Cref{thm:GeneralizedLiebThm} and the subspace argument in  \cite{gittens2011tail}, we shall be able to derive similar expectation estimates and tail bounds for the sum of arbitrary successive eigenvalues of $Y=\sum_{i=1}^mX^{(i)}$. We will leave this potential extension to future works. 

\section*{Acknowledgment}
The research was in part supported by the NSF Grant DMS-1613861. The author would like to thank Joel A. Tropp for providing deep insights and rich materials in theories of random matrices and multilinear algebra. 
The author also gratefully appreciates the inspiring discussions with Thomas Y. Hou, Florian Schaefer, Shumao Zhang and Ka Chun Lam during the development of this paper. The kind hospitality of the Erwin Schr\"odinger International Institute for Mathematics and Physics (ESI), where the early ideas of this work started, is gratefully acknowledged.

\section*{References}
\bibliographystyle{elsarticle-num}
\bibliography{reference}

\end{document}